\documentclass[a4paper,11pt,reqno]{amsart}

\usepackage[utf8]{inputenc}
\usepackage[margin=1in]{geometry}

\usepackage{amsmath}
\usepackage{amssymb}
\usepackage{amscd}
\usepackage{mathtools}
\usepackage{comment}
\mathtoolsset{showonlyrefs}

\usepackage{bbm}

\usepackage{graphicx}

\usepackage{enumitem}

\usepackage[dvipsnames]{xcolor}
\definecolor{cornellred}{rgb}{0.7, 0.11, 0.11}

\usepackage[colorlinks=true, urlcolor=RoyalBlue, citecolor=RoyalBlue, linkcolor=RoyalBlue, linktocpage, pdfpagelabels, bookmarksnumbered, bookmarksopen]{hyperref}

\newtheorem{theorem}{Theorem}[section]
\newtheorem{lemma}[theorem]{Lemma}
\newtheorem{proposition}[theorem]{Proposition}

\theoremstyle{definition}
\newtheorem{definition}[theorem]{Definition}
\newtheorem{remark}[theorem]{Remark}

\numberwithin{equation}{section}

\allowdisplaybreaks

\begin{document}

\title[Klein--Gordon--Born--Infeld equations]{Klein--Gordon--Born--Infeld equations driven \\ by mixed local--nonlocal operators}

\author[N.~Cangiotti]{Nicol\`o Cangiotti}
\author[M.~Caponi]{Maicol Caponi}
\author[A.~Maione]{Alberto Maione}
\author[E.~Vitillaro]{Enzo Vitillaro}

\address[N.~Cangiotti]{ISIS Paolino d'Aquileia\newline\indent
Via Istituto Tecnico Agrario 42\newline\indent
33043 Cividale del Friuli, Udine, Italy}
\email{nicolocangiotti@paolinodaquileia.edu.it}

\address[M.~Caponi]
{Dipartimento di Ingegneria e Scienze dell'Informazione e Matematica\newline\indent
Universit\`a degli Studi dell'Aquila\newline\indent
Via Vetoio 1, Coppito\newline\indent
67100 L'Aquila, Italy}
\email{maicol.caponi@univaq.it}

\address[A.~Maione]{Dipartimento di Matematica\newline\indent
Politecnico di Milano\newline\indent
Via Bonardi 9, Campus Leonardo\newline\indent
20133 Milan, Italy}
\email{alberto.maione@polimi.it}

\address[E.~Vitillaro]{Dipartimento di Matematica e Informatica\newline\indent
Universit\`a degli Studi di Perugia\newline\indent
Via Luigi Vanvitelli 1\newline\indent
06123 Perugia, Italy}
\email{enzo.vitillaro@unipg.it}



\begin{abstract}
In this paper we study a class of Klein--Gordon--Born--Infeld systems in which the classical Laplacian appearing in the matter equation is replaced by a mixed local--nonlocal operator, whose nonlocal part is allowed to be nonpositive definite according to a real parameter. Both the approximate and the exact Born--Infeld field equations are considered in parallel. After making the functional framework explicit, we identify a range of real parameters for which infinitely many radially symmetric standing waves exist, obtained as Mountain Pass critical points of the associated reduced energy functionals. The mixed local--nonlocal operator reduces to the classical Laplacian as the nonlocal weight vanishes, and is motivated, in the spirit of effective field theories, by nonlocal corrections to the matter dispersion relation.
\end{abstract}

\maketitle

\section{Introduction and main result} \label{sec:intro}

\subsection{The physical models}

Classical electromagnetism, as formulated by Maxwell~\cite{Maxwell1865}, is a linear theory and, as such, it assigns an infinite energy to the field generated by a point charge: the electrostatic energy density of a Coulomb field behaves like $|x|^{-4}$ near the source and therefore is not integrable. Indeed, for a point charge $e$ placed at the origin, Maxwell's electrostatics yields the Coulomb potential $\phi(x)=e/|x|$ and the field
\begin{equation}\label{eq:coulomb}
\mathbf E(x)=-\nabla\phi(x)=\frac{e}{|x|^2}\cdot\frac{x}{|x|},
\end{equation}
whose energy, obtained by integrating the density $\rho_E=\frac{1}{8\pi}|\mathbf E|^2$, is
\begin{equation}\label{eq:Wdiv}
W=\frac{e^2}{8\pi}\int_{\mathbb{R}^3}\frac{{\rm d}x}{|x|^4}=\frac{e^2}{2}\int_0^\infty\frac{{\rm d}r}{r^2}=\infty.
\end{equation}
This divergence~\eqref{eq:Wdiv} lies at the heart of the classical models of the electron.
In the classical models going back to Lorentz~\cite{Lorentz1904}, the field energy contributes to the inertial mass, 
\[
m_{\mathrm{tot}}=m_{\mathrm{mech}}+\frac{W}{c^2},
\] 
so that $W=\infty$ would force the absurd conclusion $m_{\mathrm{tot}}=\infty$.
Introducing a cutoff radius $a$, that is, truncating the energy integral~\eqref{eq:Wdiv} at $r=a$, already makes it finite, $W(a)=e^2/(2a)$; the particular value $r_0$ is then fixed by identifying the field energy with the electron rest energy, $W(r_0)=m_ec^2$, which gives $r_0=e^2/(2m_ec^2)\approx1.4\times10^{-15}\,\mathrm{m}$, half the standard \emph{classical electron radius}, whose SI expression is $r_e=e^2/(4\pi\epsilon_0 m_ec^2)\approx2.8\times10^{-15}\,\mathrm{m}$~\cite{CODATA2024}, a purely phenomenological cutoff foreign to the underlying field theory.

This picture is, moreover, in tension with experiment: to present knowledge the electron is an elementary, point-like particle, with no substructure observed down to scales far below $r_0$, in accordance with the collider compositeness searches collected by the Particle Data Group~\cite{PDG2024}. The agreement between its anomalous magnetic moment, measured in a one-electron quantum cyclotron, and quantum electrodynamics bounds any internal structure of the electron~\cite{Dehmelt1988}, while in deep-inelastic scattering at HERA the electron is treated as a point-like probe, there bounding the effective charge radius of the quarks~\cite{Chekanov2004}. 

An analogous ultraviolet pathology, now only logarithmic, resurfaces in quantum electrodynamics, where the electron self-energy is ultraviolet divergent~\cite{Weisskopf1939} and is handled only through renormalization, which absorbs the divergences into a redefinition of the physical parameters; this program runs from the covariant formulation of quantum electrodynamics~\cite{Feynman1949} to the renormalizability of gauge theories~\cite{tHooft1971}. 

The above-described long-standing inconsistency of the linear theory was the original motivation of Born and Infeld~\cite{BornInfeld1934}, who in 1934 proposed a nonlinear electrodynamics in which the field strength is bounded a priori by a parameter $b>0$.

In the purely electrostatic regime, with $\mathbf E=-\nabla\phi$, their Lagrangian density reads
\begin{equation}\label{eq:LBI}
\mathcal{L}_{BI}=\frac{b^2}{4\pi}\left(1-\sqrt{1-\frac{|\mathbf E|^2}{b^2}}\right),
\end{equation}
and the resulting theory enjoys two remarkable features: the electric field saturates at the value $b$, and the self-energy of a point charge becomes finite. We refer to Gibbons and Rasheed~\cite{GibbonsRasheed1995} for a more detailed account of this model.
The density~\eqref{eq:LBI} is real only for $|\mathbf E|\le b$; expanding the square root for $|\mathbf E|\ll b$,
\begin{equation}\label{eq:taylor}
\sqrt{1-x}=1-\frac{x}{2}-\frac{x^2}{8}+O(x^3),\quad x=\frac{|\mathbf E|^2}{b^2},
\end{equation}
gives $\mathcal{L}_{BI}=\frac{1}{8\pi}|\mathbf E|^2+\frac{1}{32\pi b^2}|\mathbf E|^4+O(b^{-4})$, whose leading term is exactly the Maxwell Lagrangian.

This theory is a genuine nonlinear extension of electromagnetism, departing from it only in the regime of strong fields.
Two structural consequences follow. First, the energy density associated with~\eqref{eq:LBI},
\begin{equation}\label{eq:rhoE}
\rho_E=\frac{b^2}{4\pi}\left(\frac{1}{\sqrt{1-|\mathbf E|^2/b^2}}-1\right),
\end{equation}
together with the nonlinear Gauss law
\begin{equation}\label{eq:gauss-nl}
\operatorname{div}\left(\frac{\mathbf E}{\sqrt{1-|\mathbf E|^2/b^2}}\right)=4\pi\rho,
\end{equation}
forces the saturation $|\mathbf E|\le b$: the energy density diverges as the field strength approaches $b$. Second, for a point charge the radial field obtained from~\eqref{eq:gauss-nl} satisfies
\begin{equation}\label{eq:asympt}
E(r)\sim\frac{e}{r^2}\quad(r\to\infty),\qquad |E(r)|\to b\quad(r\to0^+),
\end{equation}
so that the Coulomb behaviour is recovered at infinity while the field magnitude saturates at the value $b$ at the source. The corresponding self-energy is finite,
\begin{equation}\label{eq:WBI}
W_{BI}=\frac{\Gamma(1/4)^2}{6\sqrt{\pi}}\,b^{1/2}\,|e|^{3/2}<\infty.
\end{equation}
For further details see~\cite{BdAP16,GibbonsRasheed1995}. Thus, the Born--Infeld parameter $b$ regularizes the theory intrinsically, without any \emph{ad hoc} cutoff.

In the electrostatic case $\mathbf E=-\nabla\phi$, the Euler--Lagrange equation of~\eqref{eq:LBI} with a charge density $\rho$ is the \emph{exact} Born--Infeld equation
\begin{equation}\label{eq:exactBI}
-\operatorname{div}\left(\frac{\nabla\phi}{\sqrt{1-|\nabla\phi|^2/b^2}}\right)=4\pi\rho.
\end{equation}
Retaining only the first correction in~\eqref{eq:taylor} produces the approximate electrostatic Lagrangian
\begin{equation}\label{eq:approxLag}
\mathcal{L}_{BI}^{(2)}=\frac{1}{8\pi}|\nabla\phi|^2+\frac{\beta}{16\pi}|\nabla\phi|^4,\quad \beta\coloneqq\frac{1}{2b^2},
\end{equation}
whose Euler--Lagrange equation is an elliptic perturbation of the Poisson equation
\begin{equation}\label{eq:approxBI}
-\Delta\phi-\beta\Delta_4\phi=4\pi\rho,
\end{equation}
where 
\begin{equation}\label{eq:4_laplaciano}
\Delta_4\phi\coloneqq\operatorname{div}(|\nabla\phi|^2\nabla\phi).
\end{equation}
Equation~\eqref{eq:approxBI} keeps the leading nonlinear effect of the theory in moderately strong fields while remaining within reach of classical variational methods. Equation~\eqref{eq:exactBI} retains the full geometric content (and the field saturation) at the price of a singular, non-uniformly elliptic operator. Both regimes are physically meaningful, and we shall treat them in parallel.

\subsection{The mathematical problems and literature overview}

Coupling the electrostatic potential $\phi$ with a nonlinear Klein--Gordon field $\psi$ through the minimal coupling rule and looking for standing waves leads to the Klein--Gordon--Born--Infeld system, in which a nonlinear Klein--Gordon equation for $u$ is complemented by the field equation~\eqref{eq:approxBI}, or~\eqref{eq:exactBI}, with charge density $\rho=-e(\omega+e\phi)u^2$. Here, $u$ is real valued, $\omega>0$ is the frequency, $m>0$ is the mass of the field, and $e=\pm1$ is the charge.

In this paper we couple the Born--Infeld electrostatic field with matter fields depending on \emph{mixed local--nonlocal} operators. More precisely, for $s\in(0,1)$ we consider the family of operators
\begin{equation}\label{eq:Lalpha-intro}
\mathcal{L}_\alpha\coloneqq-\Delta+\alpha(-\Delta)^s,
\quad \alpha\in\mathbb{R},
\end{equation}
where $(-\Delta)^s$ denotes the fractional Laplacian. In accordance with the recent literature on mixed operators, we omit the dependence on $s$ in the notation $\mathcal{L}_\alpha$. 

Looking for standing waves
\[
\psi(x,t)=u(x)e^{i\omega t}
\]
of a nonlinear Klein--Gordon field minimally coupled with the electrostatic potential $\phi$, choosing $e=1$, and normalizing the physical constants, we are led to study, for $p\in(2,6)$ (with $6=2^*$ being the critical Sobolev exponent in dimension three), the two systems

\bigskip

\begin{itemize}
    \item \textbf{Approximate Klein--Gordon--Born--Infeld}
    \begin{equation}\label{eq:KGBI}\tag{$\mathrm{KGBI}_{\approx}$}
        \begin{cases}
        \mathcal{L}_\alpha u+m^2u-(\omega+\phi)^2u=|u|^{p-2}u &\text{in $\mathbb{R}^3$},\\[1mm]
        \Delta\phi+\beta\Delta_4\phi=4\pi(\omega+\phi)u^2 &\text{in $\mathbb{R}^3$},
        \end{cases}
    \end{equation}
    
    \bigskip
    
    \item \textbf{Exact Klein--Gordon--Born--Infeld}
    \begin{equation}\label{eq:NKGBI}\tag{$\mathrm{KGBI}$}
        \begin{cases}
        \mathcal{L}_\alpha u+m^2u-(\omega+\phi)^2u=|u|^{p-2}u &\text{in $\mathbb{R}^3$},\\[1mm]
        \operatorname{div}\left(\frac{\nabla\phi}{\sqrt{1-|\nabla\phi|^2/b^2}}\right)=4\pi(\omega+\phi)u^2 &\text{in $\mathbb{R}^3$},
        \end{cases}
    \end{equation}
\end{itemize}
where $\beta>0$, $b>0$, and $\Delta_4\phi$ is defined in~\eqref{eq:4_laplaciano}.

Over the last two decades, the variational study of Born--Infeld electrostatics and of its coupling with a matter field has been developed along two complementary lines, according to the two problems above. 

In the first one, the field equation is an elliptic perturbation of the Poisson equation; its solution is the minimizer of a smooth, strictly convex and coercive functional on a suitable reflexive Banach space. The functional groundwork for the electrostatic problem was laid by Fortunato, Orsina and Pisani~\cite{FOP02}, who introduced the natural function space for the potential and proved existence and uniqueness of the field generated by a given charge density. Building on it, d'Avenia and Pisani~\cite{DAP02} coupled the approximate Born--Infeld equation with a nonlinear Klein--Gordon field. They obtained, through an equivariant version of the Mountain Pass Theorem, infinitely many radially symmetric standing waves in equilibrium with their own electromagnetic field. Shortly afterwards, Mugnai~\cite{Mugnai04} established the same multiplicity result under general assumptions, by means of a $\mathbb{Z}_2$--Mountain Pass scheme and $L^\infty$ a priori estimates.

The second line addresses the \emph{exact} model, governed by the genuinely nonlinear, singular operator appearing in~\eqref{eq:exactBI} and~\eqref{eq:NKGBI}. The analysis is considerably more delicate. Indeed, the exact equation~\eqref{eq:exactBI} is governed by an operator that becomes singular as $|\nabla\phi|\to b$, and its natural variational formulation is a constrained minimization over the convex set $\{|\nabla\phi|\le b\}$. Producing a minimizer is comparatively easy; the difficulty lies in proving that it is a weak solution of~\eqref{eq:exactBI}: one must show that the contact set $\{|\nabla\phi|=b\}$ is Lebesgue-negligible, handle the constraint when computing variations, and upgrade the resulting variational inequality into an equation tested against the whole working space, a step that requires explicit justification.

The coupling with a nonlinear Klein--Gordon field was first treated by Yu~\cite{Yu10}, who obtained existence of least-action solitary waves of the exact system in $\mathbb{R}^3$ and on bounded domains and analyzed their profiles. The critical Sobolev exponent case was later addressed by Teng and Zhang~\cite{TZ11}, and, more recently, Shen and Zhang~\cite{SZ25} obtained a solution of the exact system as a limit of mollified problems.

On the electrostatic side, Bonheure, d'Avenia and Pomponio~\cite{BdAP16} carried out the rigorous variational study of the equation driven by an extended charge density. Working with a relaxed formulation, they proved existence, uniqueness and finiteness of the energy of the potential. The passage from minimizer to weak solution was established in the quoted paper only for radially distributed or locally bounded charge densities (the solution being classical when the charge is smooth). Bonheure, Colasuonno and F\"oldes~\cite{BCF17} considered a superposition of finitely many point charges, giving quantitative conditions on the charges that guarantee that the energy minimizer solves the Euler--Lagrange equation, and describing the singularities of the potential at the charges according to their sign. Bonheure, d'Avenia, Pomponio and Reichel~\cite{BdAPR20} developed the corresponding potential theory, proving existence of equilibrium measures for charges supported on the boundary of a domain, the uniqueness of the (interior-constant) equilibrium potential, and a characterization of balls among smooth domains. The regularity of the minimizer under suitable integrability assumptions on the charge density was then studied by Bonheure and Iacopetti~\cite{BI19}, who proved its $W^{2,2}_{\mathrm{loc}}$ regularity and, for sufficiently small data, that it is a $C^{1,\gamma}_{\mathrm{loc}}$ weak solution.

In the coupled system the difficulty is compounded by the fact that the charge density $-(\omega+\phi)u^2$ is itself part of the unknown.
Accordingly, the available existence results work within a radial setting, both for the least-action solitary waves of Yu~\cite{Yu10} and for the solutions obtained as limits of mollified problems of Shen and Zhang~\cite{SZ25}.

To the best of our knowledge, a \emph{general} existence theory for the exact coupled Klein--Gordon--Born--Infeld system in $\mathbb{R}^3$, without symmetry assumptions, is still missing. We do not settle this general question here: as in~\cite{Yu10}, we work throughout in the radially symmetric setting, which both yields the exact field as a genuine weak solution and restores the compactness needed for the multiplicity argument.

The Born--Infeld coupling has been explored in the local setting in several further directions, such as nonhomogeneous problems (Chen and Li~\cite{CL13}), the two-dimensional case with critical exponential growth (Albuquerque, Chen and Li~\cite{ACL20}), and the Schr\"odinger--Born--Infeld system (Azzollini, Pomponio and Siciliano~\cite{APS19}), none of which allows for a nonlocal kinetic term, the specific feature of the present study.

The subject remains under active investigation: very recently, De Filippis and Mingione~\cite{DFM26} have studied the electrostatic equation for rough (unbounded) charges by means of nonlinear potential theory, extending the anti-peeling theory of Bartnik and Simon (a rigidity principle for the lightlike singular set of the potential) to its sharp integrability threshold and yielding endpoint regularity of the potential and of its gradient.

In a different direction, a robust line of research has recently been devoted to \emph{mixed local--nonlocal} operators of the form~\eqref{eq:Lalpha-intro}, in which the classical second-order term coexists with a fractional one. Beyond their intrinsic analytical interest, mixed local--nonlocal operators arise naturally, in the spirit of effective field theories, as a way to encode soft nonlocal corrections to the dispersion relation of a matter field. Their analytical theory has been developed in several directions, including regularity, maximum principles, and variational methods in the possibly nonpositive definite setting; we refer the interested reader to~\cite{Ambrosio26,BDVV22,MMV23} and also~\cite{Maione26} for a general overview.

Motivated by this perspective, the authors of the present paper have recently extended the classical Benci--Fortunato theory of charged solitary waves~\cite{BenciFortunato02} to the mixed local--nonlocal setting. More precisely, they studied both the Klein--Gordon--Maxwell~\cite{CCMV23} and the Schr\"odinger--Maxwell~\cite{CCMV24} systems, allowing in both cases the nonlocal part of~\eqref{eq:Lalpha-intro} to be nonpositive definite.

The main novelty of the present work is that the local kinetic operator $-\Delta$ of the matter equation is replaced by the mixed operator $\mathcal{L}_\alpha$ defined in~\eqref{eq:Lalpha-intro}.\footnote{We borrow the term \emph{kinetic} from field theory for the part of the operator that involves derivatives of the field and governs its propagation, as opposed to a potential term depending on the field alone. In Fourier variables (see Section~\ref{sec:3}), $-\Delta$ and $(-\Delta)^s$ act on plane waves as multiplication by $|\xi|^2$ and $|\xi|^{2s}$, so $\mathcal{L}_\alpha$ modifies the dispersion relation of the matter field, $|\xi|^2\mapsto|\xi|^2+\alpha|\xi|^{2s}$; in this sense it is a kinetic, not a potential, modification, despite the nonlocality of $(-\Delta)^s$ in position space. The time derivative, absent from $\mathcal{L}_\alpha$, is carried by the frequency $\omega$ through the standing-wave ansatz $\psi=u\,e^{i\omega t}$.}

Up to the standard sign reductions one obtains the reduced systems~\eqref{eq:KGBI} and~\eqref{eq:NKGBI} stated above. The replacement of $-\Delta$ by $\mathcal{L}_\alpha$ is suggested by the same effective-theory heuristics that underlie the Born--Infeld construction: a local field equation may be viewed as the low-energy limit of a more fundamental theory whose ultraviolet completion is unknown, and integrating out the unresolved microscopic degrees of freedom (the short-distance fluctuations of the matter field, or a microscopic granularity of spacetime) can dress the kinetic term with nonlocal corrections. 
The freedom to let $\alpha$ have either sign is itself physically meaningful (indeed, it distinguishes corrections that enhance from those that soften the nonlocal contribution to the dispersion) and the delicate regime $\alpha<0$, in which $\mathcal{L}_\alpha$ ceases to be positive definite, is also mathematically interesting, as we shall underline in the next section.

The two ingredients are, moreover, \emph{complementary}: they act on two different equations of the system and modify two distinct objects. The Born--Infeld nonlinearity lives in the field equation (the nonlinear Gauss law~\eqref{eq:gauss-nl}), capping the field strength at $|\mathbf E|\le b$ and rendering the self-energy of the charge finite, whereas the mixed operator $\mathcal{L}_\alpha$ lives in the matter (Klein--Gordon) equation, reshaping the dispersion relation of the matter field. Neither can be recovered from the other, and switching off either mechanism ($b\to\infty$ or $\alpha\to0$) leaves the other fully in force.

\subsection{Main result}

We treat the approximate and the exact models in parallel, within a single variational scheme, and keep the \emph{energy space} and the \emph{solution space} of the potential carefully distinct (see Remark~\ref{rem:spaces}), a distinction that makes the reduction rigorous in the exact case.

To state our result, we introduce the function $\alpha_0\colon(0,1)\times(0,\infty)\to(0,\infty)$,
\begin{equation}\label{eq:alpha0-intro}
\alpha_0(s,t)\coloneqq s^{-s}(1-s)^{s-1}t^{1-s},\quad (s,t)\in(0,1)\times(0,\infty),
\end{equation}
which already governed the admissible range of the nonlocal weight in~\cite{CCMV23,CCMV24}, and we set
\[
\Lambda\coloneqq\max\left\{\frac{2}{p-2},1\right\},\quad p\in(2,6).
\]
We also introduce some basic notation, starting from the homogeneous spaces
\begin{equation}\label{eq:D12}
D^{1,2}(\mathbb{R}^3)\coloneqq\overline{C_c^\infty(\mathbb{R}^3)}^{\,\|\nabla\,\cdot\,\|_2},\qquad
D(\mathbb{R}^3)\coloneqq\overline{C_c^\infty(\mathbb{R}^3)}^{\,\|\nabla\,\cdot\,\|_2+\|\nabla\,\cdot\,\|_4},
\end{equation}
the former being a Hilbert space with the norm $\|u\|_{D^{1,2}}\coloneqq\|\nabla u\|_2$, the latter a Banach space with the norm $\|u\|_D\coloneqq\|\nabla u\|_2+\|\nabla u\|_4$. 
We finally set
\begin{equation}\label{sbirulino}
D_\star(\mathbb{R}^3)\coloneqq D^{1,2}(\mathbb{R}^3)\cap W^{1,\infty}(\mathbb{R}^3).
\end{equation}
Our main result reads as follows.

\begin{theorem}[\bf Existence of infinitely many standing waves]\label{thm:main}
Let $p\in(2,6)$, $s\in (0,1)$, $\alpha\in\mathbb{R}$, $m>0$, and $\omega>0$. Assume that
\[
m^2>\Lambda\omega^2\quad\text{and}\quad \alpha>-\alpha_0(s,m^2-\Lambda\omega^2).
\]
Then the following results hold.
\begin{itemize}
\item[$(i)$] \emph{\bf (Approximate Klein--Gordon--Born--Infeld)} For every $\beta>0$, system~\eqref{eq:KGBI} admits infinitely many radially symmetric weak solutions $(u_n,\phi_n)_n\subset H^1(\mathbb{R}^3)\times D(\mathbb{R}^3)$.
\item[$(ii)$] \emph{\bf (Exact Klein--Gordon--Born--Infeld)} For every $b>0$, system~\eqref{eq:NKGBI} admits infinitely many radially symmetric weak solutions $(u_n,\phi_n)_n\subset H^1(\mathbb{R}^3)\times D_\star(\mathbb{R}^3)$, with
\[
\|\nabla\phi_n\|_\infty\le b\quad\text{and}\quad\frac{|\nabla \phi_n|^2}{\sqrt{1-|\nabla \phi_n|^2/b^2}}\in L^1(\mathbb{R}^3).
\]
\end{itemize}
\end{theorem}

The two conclusions of Theorem~\ref{thm:main} share the same hypotheses on $(p,s,\alpha,m,\omega)$ and differ only in the field equation and in the space of the potential $\phi$. When $\alpha=0$, both parts of Theorem~\ref{thm:main} reduce to the classical local setting, and it is therefore natural to compare them with the existing literature. Our threshold reads
\[
m^2>\omega^2\quad\text{for }p\in[4,6),\qquad m^2>\frac{2}{p-2}\omega^2\quad\text{for }p\in(2,4).
\]
In the approximate case, part $(i)$ recovers the multiplicity results of d'Avenia--Pisani~\cite{DAP02} (established for $p\in(4,6)$ under $\omega^2<m^2$) and of Mugnai~\cite{Mugnai04} (who covered the whole range $p\in(2,6)$). Our condition $m^2>\Lambda\omega^2$ coincides exactly with the one assumed by Mugnai, and with the threshold under which D'Aprile--Mugnai~\cite{DAM04} obtained infinitely many radially symmetric solitary waves of the Klein--Gordon--Maxwell system. For the Klein--Gordon--Maxwell system, Azzollini, Pisani and Pomponio~\cite{APP11} also established the existence of a single nontrivial solitary wave in the limiting case $\omega=m$, under different assumptions on the nonlinearity.

In the exact case, part $(ii)$ \emph{improves the multiplicity result of Yu}~\cite{Yu10} by enlarging its admissible frequency range. Indeed, Yu required $(\frac{p}{2}-1)m^2>\frac{p}{2}\omega^2$, that is
\[
m^2>\frac{p}{p-2}\,\omega^2,
\]
whereas we only need
\[
m^2>\max\left\{\frac{2}{p-2},1 \right\}\omega^2.
\]
Since $\frac{p}{p-2}>\max\{\frac{2}{p-2},1\}$ for every $p\in(2,6)$, our condition is strictly weaker throughout the range. In particular, the exact Born--Infeld coupling is here shown to enjoy the very same threshold already available for the approximate coupling and for the Klein--Gordon--Maxwell system.\footnote{This improvement concerns the existence of \emph{infinitely many} solutions. For a \emph{single} radially symmetric solitary wave, wider frequency ranges for $p\in(2,4)$ are known for the exact Klein--Gordon--Born--Infeld system; see Shen and Zhang~\cite{SZ25}. To the best of our knowledge, no such refinement is available at the level of multiplicity, which is the regime we address here.}

The systems~\eqref{eq:KGBI} and~\eqref{eq:NKGBI} combine several genuinely distinct difficulties. For $\alpha\ge0$, the mixed operator $\mathcal{L}_\alpha$ is positive definite, and the bilinear form $\mathcal{B}_\alpha$ (see Section~\ref{sec:3}), after adding a positive mass term, induces a norm equivalent to the $H^1$ norm, so that the existence theory follows the classical variational pattern. Ground states and least-energy solutions for related Maxwell- and Born--Infeld-coupled problems have also been obtained by Nehari-type methods, see~\cite{AP10,WTC19}.

The interesting and substantially harder regime is $\alpha<0$,
in which the nonlocal part is in competition: $\mathcal{L}_\alpha$ is no longer positive definite, $\mathcal{B}_\alpha$ no longer induces a norm,
the spectrum of $\mathcal{L}_\alpha$ has a negative part, and even maximum principles may fail (see~\cite{BDVV22,MMV23}). The threshold $\alpha_0$ of~\eqref{eq:alpha0-intro} quantifies how negative $\alpha$ may be taken before coercivity of $\mathcal{L}_\alpha+t\,\mathrm{Id}$, with $t>0$, is lost; it arises by optimizing the interpolation estimate~\eqref{eq:interp}. For fixed $t>0$, one has $\alpha_0(s,t)\to t$ as $s\to0^+$ and $\alpha_0(s,t)\to1$ as $s\to1^-$, recovering the coercivity conditions $\alpha>-t$ and $\alpha>-1$ for the respective formal limits of $\mathcal{L}_\alpha+t\,\mathrm{Id}$, namely $-\Delta+(\alpha+t)\,\mathrm{Id}$ and $-(1+\alpha)\Delta+t\,\mathrm{Id}$, as explained in~\cite{CCMV23}.

Symmetry plays a key role in the proof of Theorem~\ref{thm:main}. Restricting to radially symmetric functions restores compactness, the embedding $H^1_r(\mathbb{R}^3)\hookrightarrow L^p(\mathbb{R}^3)$ being compact for $p\in(2,6)$, unlike on all of $H^1(\mathbb{R}^3)$, which, together with the a priori estimates, yields the Palais--Smale condition. The Principle of Symmetric Criticality then transfers the radial critical points to the
whole space.

\medskip

In the rest of the paper we fix the functional framework in Section~\ref{sec:3}, paying particular attention to the space in which the potential is sought. We then deal with the Born--Infeld equation in Section~\ref{subsec:BI}, treating the two models in parallel. Finally, Section~\ref{sec:5} is devoted to setting up the variational scheme used to prove Theorem~\ref{thm:main}.

\section{Preliminaries}\label{sec:3}
\subsection{Notation and functional setting}
This subsection collects the notation and functional framework needed for our analysis. Throughout, $\|\cdot\|_q$ denotes the norm in $L^q(\mathbb{R}^3)$, while $B_r$ denotes the ball of radius $r$ centered at the origin.
For the readers' convenience, we recall the definition of the spaces $D^{1,2}(\mathbb{R}^3)$ and $D(\mathbb{R}^3)$ given in~\eqref{eq:D12}, i.e.\
\begin{equation*}
D^{1,2}(\mathbb{R}^3)=\overline{C_c^\infty(\mathbb{R}^3)}^{\,\|\nabla\,\cdot\,\|_2},\qquad
D(\mathbb{R}^3)=\overline{C_c^\infty(\mathbb{R}^3)}^{\,\|\nabla\,\cdot\,\|_2+\|\nabla\,\cdot\,\|_4},
\end{equation*}
the former being a Hilbert space with the norm $\|u\|_{D^{1,2}}\coloneqq\|\nabla u\|_2$, the latter a Banach space with the norm $\|u\|_D\coloneqq\|\nabla u\|_2+\|\nabla u\|_4$.
We also recall that
\[
H^1(\mathbb{R}^3)\coloneqq\{u\in L^2(\mathbb{R}^3):\nabla u\in L^2(\mathbb{R}^3;\mathbb{R}^3)\},
\]
which is a Hilbert space with the norm $\|u\|_{H^1}\coloneqq(\|u\|_2^2+\|\nabla u\|_2^2)^{1/2}$.
We notice that
\[
H^1(\mathbb{R}^3)\subsetneq D^{1,2}(\mathbb{R}^3),\qquad D(\mathbb{R}^3)\subsetneq D^{1,2}(\mathbb{R}^3),
\]
so that all three spaces $D^{1,2}(\mathbb{R}^3),D(\mathbb{R}^3),H^1(\mathbb{R}^3)$ continuously embed in $L^6(\mathbb{R}^3)$ by Sobolev's inequality. 
The next proposition records the two embeddings that will be used throughout.

\begin{proposition}\label{prop:embedding}
For all $p\in[2,6]$ there exists a constant $C_p>0$ such that
\begin{equation}\label{eq:emb-H}
\|u\|_p\le C_p\|u\|_{H^1}\quad\text{for all $u\in H^1(\mathbb{R}^3)$}.
\end{equation}
Moreover, for all $q\in[6,\infty]$ there exists a constant $D_q>0$ such that
\begin{equation}\label{eq:emb-D}
\|u\|_q\le D_q\|u\|_D\quad\text{for all $u\in D(\mathbb{R}^3)$}.
\end{equation}
\end{proposition}

\begin{proof}
By Sobolev's inequality there exists $S>0$ with
\begin{equation}\label{eq:Sob}
\|u\|_6\le S\|u\|_{D^{1,2}}\quad\text{for all $u\in D^{1,2}(\mathbb{R}^3)$},
\end{equation}
which implies~\eqref{eq:emb-H}. Moreover, by~\cite[Proposition~8]{FOP02}, there exists $M>0$ with
\begin{equation}\label{eq:Minf}
\|u\|_\infty\le M\|u\|_D\quad\text{for all $u\in D(\mathbb{R}^3)$}.
\end{equation}
Combining~\eqref{eq:Sob} and~\eqref{eq:Minf}, we get~\eqref{eq:emb-D}.
\end{proof}

The spaces $D^{1,2}(\mathbb{R}^3)$ and $D(\mathbb{R}^3)$ are defined by completion. However, for the study of the field equation, it is convenient to characterize them as suitable function spaces defined by integrability conditions. This is the content of the next lemma.

\begin{lemma}\label{lem:D}
We have
\begin{align*}
D^{1,2}(\mathbb{R}^3)&=\{u\in L^6(\mathbb{R}^3):\nabla u\in L^2(\mathbb{R}^3;\mathbb{R}^3)\},\\
D(\mathbb{R}^3)&=\{u\in L^6(\mathbb{R}^3):\nabla u\in L^2(\mathbb{R}^3;\mathbb{R}^3)\cap L^4(\mathbb{R}^3;\mathbb{R}^3)\}.
\end{align*}
\end{lemma}

\begin{proof}
Denote by $\widetilde D^{1,2}(\mathbb{R}^3)$ and $\widetilde D(\mathbb{R}^3)$, respectively, the two sets on the right-hand sides. By Proposition~\ref{prop:embedding} and~\eqref{eq:Sob}, we have $D^{1,2}(\mathbb{R}^3)\subseteq\widetilde D^{1,2}(\mathbb{R}^3)$ and $D(\mathbb{R}^3)\subseteq\widetilde D(\mathbb{R}^3)$. 
Let us prove the reverse inclusions. 
We fix a family of standard mollifiers $(\rho_\varepsilon)_\varepsilon$ and a function $\eta\in C_c^\infty(\mathbb{R}^3)$ satisfying $\operatorname{supp}\eta\subset B_2$, $0\le\eta\le1$, $\eta=1$ on $B_1$ and $|\nabla\eta|\le2$. For all $R>0$, we set $\eta_R(x)\coloneqq\eta(x/R)$ for $x\in\mathbb{R}^3$.

Let us first show that $\widetilde D^{1,2}(\mathbb{R}^3)\subseteq D^{1,2}(\mathbb{R}^3)$.
Let $u\in\widetilde D^{1,2}(\mathbb{R}^3)$ be fixed. 
Then 
\[
u*\rho_\varepsilon\in C^\infty(\mathbb{R}^3)\cap\widetilde D^{1,2}(\mathbb{R}^3)
\]
and
\[
\|\nabla(u*\rho_\varepsilon)-\nabla u\|_2=\|\nabla u*\rho_\varepsilon-\nabla u\|_2\to0\quad\text{as $\varepsilon\to0$}.
\]
Therefore, we may assume $u\in C^\infty(\mathbb{R}^3)\cap\widetilde D^{1,2}(\mathbb{R}^3)$. 
Since $u\eta_R\in C_c^\infty(\mathbb{R}^3)$, the claim follows once we show that
\[
\|\nabla u-\nabla(u\eta_R)\|_2\to0\quad\text{ as }R\to\infty.
\]
We note that $\nabla(u\eta_R)=\eta_R\nabla u+u\nabla\eta_R$ and, by dominated convergence, $\|\nabla u(1-\eta_R)\|_2\to0$.
Then, it suffices to prove that 
\begin{equation}\label{eq:above}
    \|u\nabla\eta_R\|_2\to0\quad\text{ as }R\to\infty.
\end{equation}
By H\"older's inequality,
\begin{align*}
\int_{\mathbb{R}^3}|u|^2|\nabla\eta_R|^2\,{\rm d}x&=\int_{B_{2R}\setminus B_R}|u|^2|\nabla\eta_R|^2\,{\rm d}x\le\left(\int_{\mathbb{R}^3\setminus B_R}|u|^6\,{\rm d}x\right)^{1/3}\left(\int_{B_{2R}}|\nabla\eta_R|^3\,{\rm d}x\right)^{2/3}\\
&\le\left(\int_{\mathbb{R}^3\setminus B_R}|u|^6\,{\rm d}x\right)^{1/3}\frac{4}{R^2}|B_{2R}|^{2/3}=16|B_1|^{2/3}\left(\int_{\mathbb{R}^3\setminus B_R}|u|^6\,{\rm d}x\right)^{1/3}\to0,
\end{align*}
where we used that $u\in L^6(\mathbb{R}^3)$ and $|\nabla\eta_R|\le2/R$.

Let us finally show that
$\widetilde D(\mathbb{R}^3)\subseteq D(\mathbb{R}^3)$.
Let $u\in\widetilde D(\mathbb{R}^3)$ be fixed. 
As before, it holds that 
\[
u*\rho_\varepsilon\in C^\infty(\mathbb{R}^3)\cap\widetilde D(\mathbb{R}^3)
\]
and
\[
\|\nabla(u*\rho_\varepsilon)-\nabla u\|_2+\|\nabla(u*\rho_\varepsilon)-\nabla u\|_4\to0\quad\text{as $\varepsilon\to0$},
\]
so we may assume $u\in C^\infty(\mathbb{R}^3)\cap\widetilde D(\mathbb{R}^3)$. 
Then, $u\eta_R\in C_c^\infty(\mathbb{R}^3)$ and the conclusion follows once we show that
\[
\|\nabla u-\nabla(u\eta_R)\|_2\to0\quad\text{and}\quad\|\nabla u-\nabla(u\eta_R)\|_4\to0\quad\text{as $R\to\infty$}.
\]
As before, $\nabla(u\eta_R)=\eta_R\nabla u+u\nabla\eta_R$ and, by dominated convergence, $\|\nabla u(1-\eta_R)\|_2\to0$ and $\|\nabla u(1-\eta_R)\|_4\to0$.
Hence, to conclude, it suffices to prove that 
\[
\|u\nabla\eta_R\|_2\to0\quad\text{and}\quad\|u\nabla\eta_R\|_4\to0\quad\text{ as }R\to\infty.
\]
The first limit is the one just established in~\eqref{eq:above}. 
For the second, by H\"older's inequality,
\begin{align*}
\int_{\mathbb{R}^3}|u|^4|\nabla\eta_R|^4\,{\rm d}x&=\int_{B_{2R}\setminus B_R}|u|^4|\nabla\eta_R|^4\,{\rm d}x\le\left(\int_{\mathbb{R}^3\setminus B_R}|u|^6\,{\rm d}x\right)^{2/3}\left(\int_{B_{2R}}|\nabla\eta_R|^{12}\,{\rm d}x\right)^{1/3}\\
    &\le\left(\int_{\mathbb{R}^3\setminus B_R}|u|^6\,{\rm d}x\right)^{2/3}\frac{16}{R^4}|B_{2R}|^{1/3}\leq\frac{32|B_1|^{1/3}\|u\|_6^4}{R^3}\to0,
\end{align*}
since $u\in L^6(\mathbb{R}^3)$. This concludes the proof.
\end{proof}

We now turn to the potential of the \emph{exact} model. Here the field equation carries the pointwise constraint $|\nabla\phi|\le b$, which must be incorporated into the admissible set. This requires a uniform, rather than merely integrable, control of the gradient. We hence recall the space $D_\star(\mathbb{R}^3)$ introduced in~\eqref{sbirulino}, i.e.
\[
D_\star(\mathbb{R}^3)\coloneqq D^{1,2}(\mathbb{R}^3)\cap W^{1,\infty}(\mathbb{R}^3).
\]
In view of Proposition~\ref{prop:embedding} and Lemma~\ref{lem:D}, $D_\star(\mathbb{R}^3)$ can be characterized as
\[
D_\star(\mathbb{R}^3)=\{u\in L^6(\mathbb{R}^3):\nabla u\in L^2(\mathbb{R}^3;\mathbb{R}^3)\cap L^\infty(\mathbb{R}^3;\mathbb{R}^3)\}.
\]
It is a Banach space with the norm $\|u\|_{D_\star}\coloneqq\|\nabla u\|_2+\|\nabla u\|_\infty$. For all $b>0$ we introduce the sets
\begin{align}
M_b(\mathbb{R}^3)&\coloneqq\{u\in D_\star(\mathbb{R}^3):\|\nabla u\|_\infty\le b\},\label{eq:Mb}\\
K_b(\mathbb{R}^3)&\coloneqq\left\{u\in M_b(\mathbb{R}^3):\frac{|\nabla u|^2}{\sqrt{1-|\nabla u|^2/b^2}}\in L^1(\mathbb{R}^3)\right\}.\label{eq:Kb}
\end{align}
Here $M_b(\mathbb{R}^3)$ is the closed convex set on which the exact Born--Infeld energy is minimized, implicitly appearing in part ii) of our main result, while $K_b(\mathbb{R}^3)$ incorporates the additional
integrability condition satisfied by the reduced potential.

Having fixed the spaces for the potential, we now describe the matter sector, that is the fractional Sobolev scale and the mixed operator $\mathcal{L}_\alpha$ acting on it. 
It is here that the parameters $s$ and $\alpha$ enter, and that the coercivity estimate governing the admissible range of $\alpha$ is derived. 

We denote by $\mathcal{F}$ the Fourier transform, defined for $\varphi\in\mathcal{S}(\mathbb{R}^3;\mathbb{C})$ (the Schwartz space) by
\[
\mathcal{F}\varphi(\xi)=(2\pi)^{-3/2}\int_{\mathbb{R}^3}e^{-ix\cdot\xi}\varphi(x)\,{\rm d}x
\]
and extended by duality to tempered distributions. For $s\in(0,1)$ the fractional Sobolev space is
\[
H^s(\mathbb{R}^3)\coloneqq\left\{u\in L^2(\mathbb{R}^3):\int_{\mathbb{R}^3}(1+|\xi|^{2s})|\mathcal{F} u(\xi)|^2\,{\rm d}\xi<\infty\right\}.
\]
It is a Hilbert space with the norm $\|u\|_{H^s}=(\int_{\mathbb{R}^3}(1+|\xi|^{2s})|\mathcal{F} u(\xi)|^2\,{\rm d}\xi)^{1/2}$, and in addition $H^1(\mathbb{R}^3)\hookrightarrow H^s(\mathbb{R}^3)$. The fractional Laplacian is defined for $\varphi\in\mathcal{S}(\mathbb{R}^3)$ by
\[
(-\Delta)^s\varphi=\mathcal{F}^{-1}(|\xi|^{2s}\mathcal{F}\varphi),
\]
and by Plancherel's theorem
\[
\|(-\Delta)^{\frac{s}{2}}u\|_2^2=\int_{\mathbb{R}^3}|\xi|^{2s}|\mathcal{F} u(\xi)|^2\,{\rm d}\xi\quad\text{for $u\in H^s(\mathbb{R}^3)$}.
\]
In particular, for all $\varepsilon>0$ and $u\in H^1(\mathbb{R}^3)$,
\begin{equation}\label{eq:Deltas}
\|(-\Delta)^{\frac{s}{2}}u\|_2^2\le(1-s)\varepsilon^{-\frac{s}{1-s}}\|u\|_2^2+s\varepsilon\|\nabla u\|_2^2.
\end{equation}
Hence, the $s$-fractional Laplacian can be extended to a bounded linear operator
\[
(-\Delta)^s\colon H^s(\mathbb{R}^3)\to H^{-s}(\mathbb{R}^3)\coloneqq (H^s(\mathbb{R}^3))'
\]
by setting
\[
\langle (-\Delta)^s u, v\rangle_{H^{-s}\times H^s}=\int_{\mathbb{R}^3}
(-\Delta)^{\frac{s}{2}}u\,(-\Delta)^{\frac{s}{2}}v\,{\rm d}x\quad\text{for $u,v\in H^s(\mathbb{R}^3)$}.
\]
For all $\alpha\in\mathbb{R}$, the mixed local--nonlocal operator
\[
\mathcal{L}_\alpha\coloneqq -\Delta +\alpha(-\Delta)^s\colon H^1(\mathbb{R}^3)\to H^{-1}(\mathbb{R}^3)=(H^1(\mathbb{R}^3))'
\]
is defined by
\[
\langle \mathcal{L}_\alpha u, v\rangle_{H^{-1}\times H^1}=\int_{\mathbb{R}^3} \nabla u\cdot\nabla v\,{\rm d}x+\alpha\int_{\mathbb{R}^3}(-\Delta)^{\frac{s}{2}}u(-\Delta)^{\frac{s}{2}}v\,{\rm d}x\quad\text{for $u,v\in H^1(\mathbb{R}^3)$}.
\]
In particular, it is associated with the bounded bilinear form
\[
\mathcal{B}_\alpha(u,v)\coloneqq\int_{\mathbb{R}^3}\nabla u\cdot\nabla v\,{\rm d}x+\alpha\int_{\mathbb{R}^3}(-\Delta)^{\frac{s}{2}}u\,(-\Delta)^{\frac{s}{2}}v\,{\rm d}x\quad \text{for $u,v\in H^1(\mathbb{R}^3)$}.
\]
By~\eqref{eq:Deltas}, writing $\alpha^-=\max\{-\alpha,0\}$, for all $\varepsilon>0$ and $u\in H^1(\mathbb{R}^3)$,
we have
\begin{equation}\label{eq:interp}
\mathcal{B}_\alpha(u,u)\ge(1-\alpha^-s\varepsilon)\|\nabla u\|_2^2-\alpha^-(1-s)\varepsilon^{-\frac{s}{1-s}}\|u\|_2^2.
\end{equation}
The point of~\eqref{eq:interp} is that, once a mass term $t\,\|u\|_2^2$ with $t>0$ is added, the two requirements $1-\alpha^- s\varepsilon>0$ and $t-\alpha^-(1-s)\varepsilon^{-\frac{s}{1-s}}>0$ admit a common solution $\varepsilon>0$ if and only if $\alpha>-\alpha_0(s,t)$, with $\alpha_0(s,t)\coloneqq s^{-s}(1-s)^{s-1}t^{1-s}$, as in~\eqref{eq:alpha0-intro}. 
Such a mass term is produced by the coupling of the Klein--Gordon and Born--Infeld fields, and the corresponding admissible range of $\alpha$ will be studied in detail in Section~\ref{sec:5}.

\begin{remark}\label{rem:spaces}
The matter field $u$ always lives in $H^1(\mathbb{R}^3)$, or in its radial subspace $H^1_r(\mathbb{R}^3)$. The delicate choice is the space for the potential $\phi$, and it differs in the two models. In the approximate case the natural energy controls both $\|\nabla\phi\|_2$ and $\|\nabla\phi\|_4$, so $\phi$ is sought in $D(\mathbb{R}^3)$. The well-posedness of the reduction rests on the concrete characterization of $D(\mathbb{R}^3)$ given in Lemma~\ref{lem:D} and on the embedding $D(\mathbb{R}^3)\hookrightarrow L^\infty(\mathbb{R}^3)$ of~\eqref{eq:emb-D}. In the exact case the constraint $|\nabla\phi|\le b$ forces a uniform gradient bound, and the ambient space is $D_\star(\mathbb{R}^3)$, with the convex sets $M_b(\mathbb{R}^3)$ and $K_b(\mathbb{R}^3)$. Here one minimizes over $M_b(\mathbb{R}^3)$, but the minimizer solves the field equation only after the contact set $\{|\nabla\phi|=b\}$ has been shown to be negligible and the constrained variational inequality has been turned into an equation tested against all of $D_\star(\mathbb{R}^3)$. We therefore keep the \emph{energy space} and the \emph{solution space} distinct, and we verify that the reduced potential is a weak solution, in the sense made precise in Section~\ref{sec:weak}.
\end{remark}

\subsection{Weak solutions}\label{sec:weak}
In the rest of the paper $s\in (0,1)$, $\alpha\in\mathbb{R}$, $m>0$, $\omega>0$, and $p\in(2,6)$ are fixed, and we study the reduced systems~\eqref{eq:KGBI} and~\eqref{eq:NKGBI}. We conclude this section by specifying the notion of weak solution in the two cases, which differ only in the field equation and in the space of the potential.

\begin{definition}
Let $\beta>0$. A pair $(u,\phi)\in H^1(\mathbb{R}^3)\times D(\mathbb{R}^3)$ is a weak solution of~\eqref{eq:KGBI} if for all $v\in H^1(\mathbb{R}^3)$ and for all $\psi\in D(\mathbb{R}^3)$ we have
\begin{equation}
\mathcal{B}_\alpha(u,v)+m^2\int_{\mathbb{R}^3}uv\,{\rm d}x-\int_{\mathbb{R}^3}(\omega+\phi)^2uv\,{\rm d}x=\int_{\mathbb{R}^3}|u|^{p-2}uv\,{\rm d}x,\label{eq:wsol-a1}
\end{equation}
and
\begin{equation}
-\int_{\mathbb{R}^3}\nabla\phi\cdot\nabla\psi\,{\rm d}x-\beta\int_{\mathbb{R}^3}|\nabla\phi|^2\nabla\phi\cdot\nabla\psi\,{\rm d}x=4\pi\int_{\mathbb{R}^3}(\omega+\phi)\psi u^2\,{\rm d}x.
\label{eq:wsol-a2}
\end{equation}
Let $b>0$. A pair $(u,\phi)\in H^1(\mathbb{R}^3)\times K_b(\mathbb{R}^3)$ is a weak solution of~\eqref{eq:NKGBI} if~\eqref{eq:wsol-a1} holds for all $v\in H^1(\mathbb{R}^3)$ and for all $\psi\in D_\star(\mathbb{R}^3)$ one has
\begin{equation}\label{eq:wsol-b2}
-\int_{\mathbb{R}^3}\frac{\nabla\phi\cdot\nabla\psi}{\sqrt{1-|\nabla\phi|^2/b^2}}\,{\rm d}x=4\pi\int_{\mathbb{R}^3}(\omega+\phi)\psi u^2\,{\rm d}x.
\end{equation}
\end{definition}

\begin{remark}
\label{rem:wd}
All the integrals above are well defined, as we are going to explain in the sequel. Since $H^1(\mathbb{R}^3)\hookrightarrow L^p(\mathbb{R}^3)$ for $p\in[2,6]$ and $D(\mathbb{R}^3)\hookrightarrow L^\infty(\mathbb{R}^3)$, the terms in~\eqref{eq:wsol-a1}--\eqref{eq:wsol-a2} make sense for $u,v\in H^1(\mathbb{R}^3)$ and $\phi,\psi\in D(\mathbb{R}^3)$. In the exact case, if $\phi\in K_b(\mathbb{R}^3)$ then $4\pi(\omega+\phi)u^2\in L^1(\mathbb{R}^3)$, the set $\{|\nabla\phi|=b\}$ is Lebesgue-negligible by the definition of $K_b(\mathbb{R}^3)$, and for every $\psi\in D_\star(\mathbb{R}^3)$, splitting $\mathbb{R}^3$ into the sets $\{|\nabla\phi|\le b/2\}$ and $\{|\nabla\phi|> b/2\}$,
\[
\int_{\mathbb{R}^3}\frac{|\nabla\psi|^2}{\sqrt{1-|\nabla\phi|^2/b^2}}\,{\rm d}x\le\frac{2}{\sqrt3}\|\nabla\psi\|_2^2+\frac{4}{b^2}\|\nabla\psi\|_\infty^2\int_{\mathbb{R}^3}\frac{|\nabla\phi|^2}{\sqrt{1-|\nabla\phi|^2/b^2}}\,{\rm d}x<\infty,
\]
so that, by H\"older's inequality, the left-hand side of~\eqref{eq:wsol-b2} is finite.
\end{remark}

Both systems~\eqref{eq:KGBI} and~\eqref{eq:NKGBI} are solved by the same three steps. 
First, for every $u\in L^2(\mathbb{R}^3)$, a unique potential $\Phi(u)$ is obtained by minimizing a Born--Infeld energy, with suitable bounds and continuity
properties. It solves the field equation in the approximate case and, for radial $u$, in the exact case (Section~\ref{subsec:BI}). This is the step where the two models genuinely differ, and where the exact case requires the contact-set analysis of Remark~\ref{rem:spaces}. Second, inserting $\phi=\Phi(u)$ reduces the coupled system to the search for critical points of a single functional $J$ on $H^1(\mathbb{R}^3)$ (Subsection~\ref{sec:5.1}). Third, restricting $J$ to the radial subspace, we verify the geometric and compactness assumptions of an equivariant Mountain Pass Theorem and conclude (Subsection~\ref{subsec:proof}). The second and third steps are common to the two models, and we carry them out once.

\section{The Born--Infeld equation}\label{subsec:BI}

This is the only step in which the approximate and the exact models require genuinely different arguments, and we treat them in turn.

\subsection{The approximate equation}\label{subsubsec:BI-approx}

For $\beta>0$ and $u\in L^2(\mathbb{R}^3)$ we define the energy functional $E_u\colon D(\mathbb{R}^3)\to\mathbb{R}$ by
\begin{equation}\label{eq:E}
E_u(\phi)\coloneqq\frac{1}{8\pi}\|\nabla\phi\|_2^2+\frac{\beta}{16\pi}\|\nabla\phi\|_4^4+\frac{1}{2}\int_{\mathbb{R}^3}(\omega+\phi)^2u^2\,{\rm d}x.
\end{equation}
The following result shows some relevant properties of solutions of the approximate equation.
\begin{lemma}\label{lem:Phiu}
Let $\beta>0$. For all $u\in L^2(\mathbb{R}^3)$ there exists a unique $\Phi(u)\in D(\mathbb{R}^3)$ solving~\eqref{eq:wsol-a2}, namely the unique minimizer of $E_u$. Moreover, we have
\begin{equation}
\label{eq:Phi-est}
-\omega\le\Phi(u)\le0\quad\text{a.e.\ in $\mathbb{R}^3$.}
\end{equation}
Finally, the map $\Phi\colon L^2(\mathbb{R}^3)\to D(\mathbb{R}^3)$ is continuous.
\end{lemma}

\begin{proof}
Let $u\in L^2(\mathbb{R}^3)$ be fixed. 
The functional $E_u$ is of class $C^1$ on $D(\mathbb{R}^3)$ and convex, so that
\[
E_u(\phi_2)\ge E_u(\phi_1)+E_u'(\phi_1)(\phi_2-\phi_1)\quad\text{for all $\phi_1,\phi_2\in D(\mathbb{R}^3)$}\,.
\]
Hence, $\phi$ minimizes $E_u$ if and only if $E_u'(\phi)[\psi]=0$ for all $\psi\in D(\mathbb{R}^3)$, that is if and only if~\eqref{eq:wsol-a2} holds. 
Since one has
\[
E_u(\phi)\ge\frac{1}{8\pi}\|\nabla\phi\|_2^2+\frac{\beta}{16\pi}\|\nabla\phi\|_4^4\quad\text{for all $\phi\in D(\mathbb{R}^3)$},
\]
the functional $E_u$ is strictly convex and coercive on $D(\mathbb{R}^3)$.
Then, by the direct method, $E_u$ has a unique minimizer $\Phi(u)$, which is the unique solution of~\eqref{eq:wsol-a2}.

Let us show the validity of~\eqref{eq:Phi-est}. 
We first show that
\begin{equation}\label{eq:tilde}
\widetilde\Phi(u)\coloneqq\min\{\max\{\Phi(u),-\omega\},0\}
\in D(\mathbb{R}^3).
\end{equation}
By definition, $\widetilde\Phi(u)\leq 0$ a.e.\ in $\mathbb{R}^3$, and
\begin{equation}\label{eq:tilde_1}
    \widetilde\Phi(u)=
\begin{cases}
\Phi(u)&\text{a.e.\ on $\{-\omega\le \Phi(u)\le 0\}$},\\
0&\text{a.e.\ on $\{\Phi(u)>0\}$},\\
-\omega&\text{a.e.\ on $\{\Phi(u)<-\omega\}$}.
\end{cases}
\end{equation}
Moreover, by the Sobolev chain rule for truncations, it holds that
\begin{equation}\label{eq:tilde_2}
    \nabla\widetilde\Phi(u)=
\begin{cases}
\nabla\Phi(u)&\text{a.e.\ on $\{-\omega\le \Phi(u)\le 0\}$},\\
0&\text{a.e.\ on $\{\Phi(u)>0\}\cup \{\Phi(u)<-\omega\}$}.
\end{cases}
\end{equation}
In particular, by~\eqref{eq:tilde_1} and~\eqref{eq:tilde_2},
\begin{equation}\label{eq:tilde1}
|\widetilde\Phi(u)|\le |\Phi(u)|\quad\text{and}\quad|\nabla\widetilde\Phi(u)|\le |\nabla\Phi(u)|
\quad\text{a.e.\ in $\mathbb{R}^3$}.
\end{equation}
By the characterization of $D(\mathbb{R}^3)$ given in
Lemma~\ref{lem:D}, the estimates~\eqref{eq:tilde1}
imply~\eqref{eq:tilde}. 
Furthermore, by~\eqref{eq:tilde_1}
\[
\omega+\widetilde\Phi(u)=
\begin{cases}
\omega+\Phi(u)&\text{a.e.\ on $\{-\omega\le \Phi(u)\le 0\}$},\\
\omega&\text{a.e.\ on $\{\Phi(u)>0\}$},\\
0&\text{a.e.\ on $\{\Phi(u)<-\omega\}$}.
\end{cases}
\]
Note that $\omega+\widetilde\Phi(u)\geq0$ a.e.\ in $\mathbb{R}^3$. 
In particular,
\begin{equation}\label{eq:tilde2}
|\omega+\widetilde\Phi(u)|
=\omega+\widetilde\Phi(u)
\le |\omega+\Phi(u)|
\quad\text{a.e.\ in $\mathbb{R}^3$}.
\end{equation}
Combining~\eqref{eq:tilde1} and~\eqref{eq:tilde2}, we obtain
\[
E_u(\widetilde\Phi(u))\le E_u(\Phi(u)).
\]
Since $\Phi(u)$ is the unique minimizer of $E_u$, we conclude that
\[
\widetilde\Phi(u)=\Phi(u)\quad\text{a.e.\ in }\mathbb{R}^3,
\]
and, by~\eqref{eq:tilde_1},~\eqref{eq:Phi-est} holds.

Finally, we are going to prove that $\Phi\colon L^2(\mathbb{R}^3)\to D(\mathbb{R}^3)$ is continuous. Let $u_n\to u$ in $L^2(\mathbb{R}^3)$ as $n\to\infty$. Consequently, $C_1\coloneqq\sup_n\|u_n\|_2<\infty$. Testing~\eqref{eq:wsol-a2} first for $\Phi(u_n)$ and then for $\Phi(u)$, in both cases with $\psi=\Phi(u_n)-\Phi(u)$, we obtain
\begin{align*}
&\int_{\mathbb{R}^3}\nabla\Phi(u_n)\cdot(\nabla\Phi(u_n)-\nabla\Phi(u))\,{\rm d}x+\beta\int_{\mathbb{R}^3}|\nabla\Phi(u_n)|^2\nabla\Phi(u_n)\cdot(\nabla\Phi(u_n)-\nabla\Phi(u))\,{\rm d}x\\
&=-4\pi\omega\int_{\mathbb{R}^3}(\Phi(u_n)-\Phi(u))u_n^2\,{\rm d}x-4\pi\int_{\mathbb{R}^3}\Phi(u_n)(\Phi(u_n)-\Phi(u))u_n^2\,{\rm d}x,
\end{align*}
and
\begin{align*}
&\int_{\mathbb{R}^3}\nabla\Phi(u)\cdot(\nabla\Phi(u_n)-\nabla\Phi(u))\,{\rm d}x+\beta\int_{\mathbb{R}^3}|\nabla\Phi(u)|^2\nabla\Phi(u)\cdot(\nabla\Phi(u_n)-\nabla\Phi(u))\,{\rm d}x\\
&=-4\pi\omega\int_{\mathbb{R}^3}(\Phi(u_n)-\Phi(u))u^2\,{\rm d}x-4\pi\int_{\mathbb{R}^3}\Phi(u)(\Phi(u_n)-\Phi(u))u^2\,{\rm d}x.
\end{align*}
Subtracting the second identity from the first and using the monotonicity inequality
\[
(|\xi_1|^2\xi_1-|\xi_2|^2\xi_2)\cdot(\xi_1-\xi_2)\ge\frac{1}{4}|\xi_1-\xi_2|^4\quad\text{for all $\xi_1,\xi_2\in\mathbb{R}^3$},
\]
together with~\eqref{eq:Phi-est} (so that $0\le\omega+\Phi(u)\le\omega$ and $|\Phi(u_n)-\Phi(u)|\le\omega$), we get
\begin{align*}
&\|\nabla\Phi(u_n)-\nabla\Phi(u)\|_2^2+\frac{\beta}{4}\|\nabla\Phi(u_n)-\nabla\Phi(u)\|_4^4\\
&\le4\pi\omega\int_{\mathbb{R}^3}(\Phi(u_n)-\Phi(u))(u^2-u_n^2)\,{\rm d}x+4\pi\int_{\mathbb{R}^3}(\Phi(u)u^2-\Phi(u_n)u_n^2)(\Phi(u_n)-\Phi(u))\,{\rm d}x\\
&=4\pi\int_{\mathbb{R}^3}(\omega+\Phi(u))(\Phi(u_n)-\Phi(u))(u^2-u_n^2)\,{\rm d}x-4\pi\int_{\mathbb{R}^3}(\Phi(u_n)-\Phi(u))^2u_n^2\,{\rm d}x\\
&\le4\pi\int_{\mathbb{R}^3}(\omega+\Phi(u))(\Phi(u_n)-\Phi(u))(u^2-u_n^2)\,{\rm d}x\\
&\le4\pi\omega^2\|u+u_n\|_2\|u-u_n\|_2\le8\pi\omega^2C_1\|u-u_n\|_2.
\end{align*}
Since $u_n\to u$ in $L^2(\mathbb{R}^3)$, we conclude that $\Phi(u_n)\to\Phi(u)$ in $D(\mathbb{R}^3)$, so proving the announced continuity.
\end{proof}

\subsection{The exact equation}
\label{subsubsec:BI-exact}

For $b>0$ and $u\in L^2(\mathbb{R}^3)$, we define the Born--Infeld energy $F_u\colon M_b(\mathbb{R}^3)\to[0,\infty)$ by
\begin{equation}\label{eq:E2}
F_u(\phi)\coloneqq\frac{b^2}{4\pi}\int_{\mathbb{R}^3}\left(1-\sqrt{1-\frac{|\nabla\phi|^2}{b^2}}\right){\rm d}x+\frac{1}{2}\int_{\mathbb{R}^3}(\omega+\phi)^2u^2\,{\rm d}x.
\end{equation}
 We are going to use the elementary inequality
\begin{equation}\label{eq:f-bound}
\frac{t}{2}\le1-\sqrt{1-t}\le t\quad\text{for all $t\in[0,1]$},
\end{equation}
which follows from $1-\frac{t}{2}\ge\sqrt{1-t}\Leftrightarrow\frac{t^2}{4}\ge0$ and $1-t\le\sqrt{1-t}\Leftrightarrow t(t-1)\le0$. By it we get $F_u(\phi)<\infty$ for all $\phi\in M_b(\mathbb{R}^3)$ and $F_u(\phi)\ge\frac{1}{8\pi}\|\nabla\phi\|_2^2$.

By following ~\cite{Yu10}, we are now going to establish the variational properties of $F_u$.

\begin{lemma}\label{lem:weak-closed}
Let $b>0$. The set $M_b(\mathbb{R}^3)$ is a closed convex subset of $D^{1,2}(\mathbb{R}^3)$, hence it is also weakly closed.
\end{lemma}

\begin{proof}
Convexity is clear. If $(\psi_n)_n\subset M_b(\mathbb{R}^3)$ and $\psi_n\to\psi$ in $D^{1,2}(\mathbb{R}^3)$ as $n\to\infty$, then $\nabla\psi_n\to\nabla\psi$ a.e.\ along a subsequence. Since $|\nabla\psi_n|\le b$ a.e., also $|\nabla\psi|\le b$ a.e., so $\psi\in M_b(\mathbb{R}^3)$.
\end{proof}

\begin{lemma}\label{lem:conv-unif}
Let $b>0$, $(\phi_n)_n\subset M_b(\mathbb{R}^3)$, and $\phi\in M_b(\mathbb{R}^3)$ with $\phi_n\rightharpoonup\phi$ in $D^{1,2}(\mathbb{R}^3)$ as $n\to\infty$. Then $\phi_n\to\phi$ in $L^\infty_{\mathrm{loc}}(\mathbb{R}^3)$ as $n\to\infty$.
\end{lemma}

\begin{proof}
The sequence $(\phi_n)_n$ is bounded in $D_\star(\mathbb{R}^3)$, hence in $D(\mathbb{R}^3)$. Therefore, by Proposition~\ref{prop:embedding}, it is bounded in $W^{1,\infty}(\mathbb{R}^3)$ as well. Hence it is locally uniformly equicontinuous and bounded. By the Arzel\`a--Ascoli theorem we thus get $\phi_n\to\phi$ in $L^\infty_{\mathrm{loc}}(\mathbb{R}^3)$.
\end{proof}

\begin{lemma}\label{lem:weak-semi}
Let $b>0$ and $u\in L^2(\mathbb{R}^3)$. The functional $F_u$ is coercive and sequentially weakly lower semicontinuous in $D^{1,2}(\mathbb{R}^3)$.
\end{lemma}

\begin{proof}
The coercivity follows from the estimate $F_u(\phi)\ge\frac{1}{8\pi}\|\nabla\phi\|_2^2$. Let $\phi_n\rightharpoonup\phi$ in $D^{1,2}(\mathbb{R}^3)$ as $n\to\infty$, with $(\phi_n)_n,\phi\in M_b(\mathbb{R}^3)$. Since $\xi\mapsto1-\sqrt{1-|\xi|^2/b^2}$ is convex on $\{|\xi|\le b\}$, we have
\[
\int_{\mathbb{R}^3}\left(1-\sqrt{1-\frac{|\nabla\phi|^2}{b^2}}\right){\rm d}x\le\liminf_{n\to\infty}\int_{\mathbb{R}^3}\left(1-\sqrt{1-\frac{|\nabla\phi_n|^2}{b^2}}\right){\rm d}x,
\]
while by Lemma~\ref{lem:conv-unif} and dominated convergence we get $\int_{\mathbb{R}^3}(\omega+\phi_n)^2u^2\,{\rm d}x\to\int_{\mathbb{R}^3}(\omega+\phi)^2u^2\,{\rm d}x$. Hence $F_u$ is sequentially weakly lower semicontinuous.
\end{proof}

\begin{proposition}\label{prop:exist-min}
Let $b>0$ and $u\in L^2(\mathbb{R}^3)$. There exists a unique minimizer $\Phi(u)\in M_b(\mathbb{R}^3)$ of $F_u$.
\end{proposition}

\begin{proof}
Existence follows from Lemmas~\ref{lem:weak-closed} and~\ref{lem:weak-semi} by the direct method, and uniqueness from the strict convexity of $\xi\mapsto1-\sqrt{1-|\xi|^2/b^2}$.
\end{proof}

The next proposition is the crucial step discussed in Remark~\ref{rem:spaces}: it shows that the contact set is negligible and turns minimality into a variational inequality. Its proof follows~\cite[Proposition~2.7]{BdAP16}.

\begin{proposition}\label{prop:uniq}
Let $b>0$ and $u\in L^2(\mathbb{R}^3)$. Let $\Phi(u)\in M_b(\mathbb{R}^3)$ be the minimizer of $F_u$. Then, the set $A\coloneqq\{x\in\mathbb{R}^3:|\nabla\Phi(u)(x)|=b\}$ has zero Lebesgue measure and we have
\begin{equation}\label{eq:estimate}
\int_{\mathbb{R}^3}\frac{|\nabla\Phi(u)|^2}{\sqrt{1-|\nabla\Phi(u)|^2/b^2}}\,{\rm d}x\le-4\pi\int_{\mathbb{R}^3}(\omega+\Phi(u))\Phi(u)u^2\,{\rm d}x.
\end{equation}
In particular, $\Phi(u)\in K_b(\mathbb{R}^3)$. Moreover, $\phi\in K_b(\mathbb{R}^3)$ minimizes $F_u$ if and only if
\begin{equation}\label{eq:var-ineq}
\int_{\mathbb{R}^3}\frac{|\nabla\phi|^2}{\sqrt{1-|\nabla\phi|^2/b^2}}\,{\rm d}x-\int_{\mathbb{R}^3}\frac{\nabla\phi\cdot\nabla\psi}{\sqrt{1-|\nabla\phi|^2/b^2}}\,{\rm d}x\le4\pi\int_{\mathbb{R}^3}(\omega+\phi)(\psi-\phi)u^2\,{\rm d}x
\end{equation}
for all $\psi\in M_b(\mathbb{R}^3)$.
\end{proposition}

\begin{proof}
Since $0\in M_b(\mathbb{R}^3)$, we have $(1-t)\Phi(u)\in M_b(\mathbb{R}^3)$ for all $t\in[0,1]$, and by minimality $F_u(\Phi(u))\le F_u((1-t)\Phi(u))$, i.e.,
\begin{align}
&b^2\int_{\mathbb{R}^3}\left(\sqrt{1-\frac{(1-t)^2|\nabla\Phi(u)|^2}{b^2}}-\sqrt{1-\frac{|\nabla\Phi(u)|^2}{b^2}}\right){\rm d}x\nonumber\\
&\le-2\pi t\int_{\mathbb{R}^3}(2\omega+2\Phi(u)-t\Phi(u))\Phi(u)u^2\,{\rm d}x.\label{eq:sqrt}
\end{align}
By splitting the first integral over $\mathbb{R}^3\setminus A$ and $A$, since both resulting terms are non-negative, we get
\[
b^2|A|\sqrt{2t-t^2}\le-2\pi t\int_{\mathbb{R}^3}(2\omega+2\Phi(u)-t\Phi(u))\Phi(u)u^2\,{\rm d}x\quad\text{for $t\in(0,1]$,}
\]
that is
\[
b^2|A|\le-2\pi\sqrt\frac{t}{2-t}\int_{\mathbb{R}^3}(2\omega+2\Phi(u)-t\Phi(u))\Phi(u)u^2\,{\rm d}x\quad\text{for all $t\in(0,1]$}.
\]
Then, by letting $t\to0^+$, we get that $|A|=0$. Moreover, by dividing inequality~\eqref{eq:sqrt} by $t$, we have
\begin{align*}
&\int_{\mathbb{R}^3}\frac{(2-t)|\nabla\Phi(u)|^2}{\left(\sqrt{1-(1-t)^2|\nabla\Phi(u)|^2/b^2}+\sqrt{1-|\nabla\Phi(u)|^2/b^2}\right)}{\rm d}x\\
&\le-2\pi\int_{\mathbb{R}^3}(2\omega+2\Phi(u)-t\Phi(u))\Phi(u)u^2\,{\rm d}x.
\end{align*}
By Fatou's lemma, the former inequality yields~\eqref{eq:estimate} as $t\to0^+$.
In particular, $\Phi(u)\in K_b(\mathbb{R}^3)$.

Let now $\Phi(u)$ be the minimizer of $F_u$ and fix $\psi\in M_b(\mathbb{R}^3)$. Since $\Phi(u)+t(\psi-\Phi(u))\in M_b(\mathbb{R}^3)$ for $t\in[0,1]$, minimality gives, for $t\in(0,1]$,
\begin{align*}
&\int_{\mathbb{R}^3}\frac{(2-t)|\nabla\Phi(u)|^2-2(1-t)\nabla\Phi(u)\cdot\nabla\psi-t|\nabla\psi|^2}{\sqrt{1-|\nabla\Phi(u)+t(\nabla\psi-\nabla\Phi(u))|^2/b^2}+\sqrt{1-|\nabla\Phi(u)|^2/b^2}}\,{\rm d}x\\
&\le 2\pi\int_{\mathbb{R}^3}(2\omega+2\Phi(u)+t(\psi-\Phi(u)))(\psi-\Phi(u))u^2\,{\rm d}x.
\end{align*}
The integrand on the left is dominated by $3\left(|\nabla\Phi(u)|^2+|\nabla\psi|^2\right)\big/\sqrt{1-|\nabla\Phi(u)|^2/b^2}\in L^1(\mathbb{R}^3)$, since $\Phi(u)\in K_b(\mathbb{R}^3)$ and $\psi\in M_b(\mathbb{R}^3)$. By dominated convergence as $t\to0^+$ we obtain~\eqref{eq:var-ineq}.
Conversely, if $\phi\in K_b(\mathbb{R}^3)$ satisfies~\eqref{eq:var-ineq}, the convexity of $\xi\mapsto1-\sqrt{1-|\xi|^2/b^2}$ gives, for all $\psi\in M_b(\mathbb{R}^3)$,
\[
F_u(\psi)\ge F_u(\phi)+\frac{1}{4\pi}\int_{\mathbb{R}^3}\frac{\nabla\phi\cdot(\nabla\psi-\nabla\phi)}{\sqrt{1-|\nabla\phi|^2/b^2}}\,{\rm d}x+\int_{\mathbb{R}^3}(\omega+\phi)(\psi-\phi)u^2\,{\rm d}x\ge F_u(\phi).
\]
Thus, $\phi$ is the unique minimizer of $F_u$.
\end{proof}

\begin{lemma}\label{lem:Phiu2}
Let $b>0$ and $u\in L^2(\mathbb{R}^3)$. Let $\Phi(u)\in K_b(\mathbb{R}^3)$ be the minimizer of $F_u$. We have
\begin{equation}\label{eq:Phi-est2}
-\omega\le\Phi(u)\le 0\quad\text{a.e.\ in $\mathbb{R}^3$.}
\end{equation}
Moreover, the map $\Phi\colon L^2(\mathbb{R}^3)\to L^\infty_{\mathrm{loc}}(\mathbb{R}^3)$ is continuous.
\end{lemma}

\begin{proof}
To prove~\eqref{eq:Phi-est2}, we argue as in the proof of Lemma~\ref{lem:Phiu}. Indeed, the truncation
\[
\widetilde\Phi(u)\coloneqq
\min\{\max\{\Phi(u),-\omega\},0\}
\]
belongs to $M_b(\mathbb{R}^3)$ and satisfies
\[
|\nabla\widetilde\Phi(u)|\le|\nabla\Phi(u)|
\quad\text{and}\quad
|\omega+\widetilde\Phi(u)|\le|\omega+\Phi(u)|
\quad\text{a.e.\ in $\mathbb{R}^3$}.
\]
Consequently,
\[
F_u(\widetilde\Phi(u))\le F_u(\Phi(u)).
\]
Since $\Phi(u)$ is the unique minimizer of $F_u$, we conclude that
$\widetilde\Phi(u)=\Phi(u)$ a.e.\ in $\mathbb{R}^3$,
which implies~\eqref{eq:Phi-est2}.

By~\eqref{eq:estimate}, for all $u\in L^2(\mathbb{R}^3)$ we have
\begin{align}
\|\nabla\Phi(u)\|_2^2&\le\int_{\mathbb{R}^3}\frac{|\nabla\Phi(u)|^2}{\sqrt{1-|\nabla\Phi(u)|^2/b^2}}\,{\rm d}x\nonumber\\
&\le-4\pi\omega\int_{\mathbb{R}^3}\Phi(u)u^2\,{\rm d}x-4\pi\int_{\mathbb{R}^3}\Phi(u)^2u^2\,{\rm d}x\nonumber\\
&\le-4\pi\omega\int_{\mathbb{R}^3}\Phi(u)u^2\,{\rm d}x\le4\pi\omega^2 \|u\|_2^2.\label{eq:L2-bound}
\end{align}
Let now $u_n\to u$ in $L^2(\mathbb{R}^3)$ as $n\to\infty$. By~\eqref{eq:L2-bound} the sequence $(\Phi(u_n))_n$ is bounded in $D^{1,2}(\mathbb{R}^3)$, so up to a subsequence $\Phi(u_n)\rightharpoonup\phi$ in $D^{1,2}(\mathbb{R}^3)$ as $n\to\infty$. By Lemma~\ref{lem:weak-closed} we get $\phi\in M_b(\mathbb{R}^3)$, and by Lemma~\ref{lem:conv-unif} we obtain $\Phi(u_n)\to\phi$ in $L^\infty_{\mathrm{loc}}(\mathbb{R}^3)$. It remains to show $\phi=\Phi(u)$. For every $\psi\in M_b(\mathbb{R}^3)$ we have $F_{u_n}(\Phi(u_n))\le F_{u_n}(\psi)$. By dominated convergence $F_{u_n}(\psi)\to F_u(\psi)$, while arguing as in Lemma~\ref{lem:weak-semi} gives $F_u(\phi)\le\liminf_n F_{u_n}(\Phi(u_n))$. Hence $F_u(\phi)\le F_u(\psi)$ for all $\psi\in M_b(\mathbb{R}^3)$, and by uniqueness of the minimizer we get $\phi=\Phi(u)$.
\end{proof}

For $R\in\operatorname{O}(3)$ and $v\colon\mathbb{R}^3\to\mathbb{R}$ we write
\begin{equation}\label{eq:vR}
v_R(x)\coloneqq v(Rx)\quad\text{for $x\in \mathbb{R}^3$},
\end{equation}
and we recall that $v$ is radially symmetric if and only if $v_R=v$ for all $R\in\operatorname{O}(3)$. The following result, whose proof follows~\cite[Theorem~1.4]{BdAP16}, is where radial symmetry is used to turn the minimizer into a weak solution.

\begin{proposition}
\label{prop:BI-weak-sol}
Let $b>0$ and let $u\in L^2(\mathbb{R}^3)$ be radially symmetric. Let $\Phi(u)\in K_b(\mathbb{R}^3)$ be the minimizer of $F_u$. Then $\Phi(u)$ is radially symmetric and is a weak solution of~\eqref{eq:wsol-b2}.
\end{proposition}

\begin{proof}
Since $u$ is radial, for every $R\in\operatorname{O}(3)$ the function $(\Phi(u))_R$ is the minimizer of $F_{u_R}=F_u$, so by uniqueness $(\Phi(u))_R=\Phi(u)$, namely $\Phi(u)$ is radially symmetric. In particular, there exists $\phi_u \colon [0,\infty)\to \mathbb{R}$ such that
\[
\Phi(u)(x)=\phi_u(|x|)\quad\text{for $x\in \mathbb{R}^3$}.
\]
Therefore, we get
\[
|\nabla \Phi(u)(x)|=|\phi_u'(|x|)|\quad\text{for a.e.\ $x\in \mathbb{R}^3$}.
\]

For $n\in\mathbb{N}$ we set $A_n\coloneqq\{r\in(0,\infty):|\phi_u'(r)|\ge b-\frac{1}{n}\}$. Since $A\coloneqq\{r:|\phi_u'(r)|=b\}$ has zero measure by Proposition~\ref{prop:uniq} and $\bigcap_n A_n=A$, we have $\chi_{A_n}\to0$ a.e.\ as $n\to\infty$. Let $\Psi\in C_c^\infty(\mathbb{R}^3)$ be radial, that is $\Psi(x)=\psi(|x|)$, and define
\[
\psi_n(r)\coloneqq-\int_r^\infty\psi'(s)[1-\chi_{A_n}(s)]\,{\rm d}s\quad\text{for $r\in [0,\infty)$},\qquad\Psi_n(x)\coloneqq\psi_n(|x|)\quad \text{for $x\in\mathbb{R}^3$}.
\]
By construction, $\Psi_n\in W^{1,\infty}(\mathbb{R}^3)$ has compact support. Thus $\Psi_n\in D_\star(\mathbb{R}^3)$. Moreover, we claim that, if $|t|$ is sufficiently small, then
\[
\Phi(u) + t\Psi_n \in M_b(\mathbb{R}^3).
\]
Indeed, since
\[
(\phi_u+t\psi_n)'=\phi_u' + t\psi'[1-\chi_{A_n}],
\]
then, we have
\[
|(\phi_u+t\psi_n)'(r)| = |\phi_u'(r)| \leq b\quad\text{for all $r\in A_n$}.
\]
Otherwise, taking
\[
|t| \leq \frac{1}{n\|\psi'\|_\infty},
\]
we obtain
\[
|(\phi_u+t\psi_n)'(r)|\le|\phi_u'(r)|+|t|\,\|\psi'\|_\infty<b-\frac{1}{n}+|t|\,\|\psi'\|_\infty\le b,
\]
so proving our claim.

Arguing as in Proposition~\ref{prop:uniq}, since $\Phi(u)$ minimizes $F_u$, we get, for every $n$,
\begin{equation}
\label{eq:rad1}
\int_0^\infty\frac{\phi_u'\psi'}{\sqrt{1-|\phi_u'|^2/b^2}}[1-\chi_{A_n}]\,r^2\,{\rm d}r+\int_{\mathbb{R}^3}(\omega+\Phi(u))\Psi_n u^2\,{\rm d}x=0.
\end{equation}
By dominated convergence the first integral tends to
\[
\int_0^\infty\frac{\phi_u'\psi'}{\sqrt{1-|\phi_u'|^2/b^2}}\,r^2\,{\rm d}r \quad \text{as $n \to \infty$},
\]
while $\Psi_n\to\Psi$ in $L^\infty(\mathbb{R}^3)$ gives
\[
\int_{\mathbb{R}^3}(\omega+\Phi(u))\Psi_n u^2\,{\rm d}x\to\int_{\mathbb{R}^3}(\omega+\Phi(u))\Psi u^2\,{\rm d}x \quad \text{as $n \to \infty$}.
\]
Passing to the limit in~\eqref{eq:rad1} we get
\begin{equation}\label{eq:rad2}
\int_{\mathbb{R}^3}\frac{\nabla\Phi(u)\cdot\nabla\Psi}{\sqrt{1-|\nabla\Phi(u)|^2/b^2}}\,{\rm d}x=-4\pi\int_{\mathbb{R}^3}(\omega+\Phi(u))\Psi u^2\,{\rm d}x,
\end{equation}
for all radial $\Psi\in C_c^\infty(\mathbb{R}^3)$. By approximation with radial functions in $C_c^\infty(\mathbb R^3)$, converging in $D^{1,2}(\mathbb R^3)$ and uniformly, with uniformly bounded gradients, and arguing as in the proof of~\cite[Lemma~2.10]{BdAP16}, equation~\eqref{eq:rad2} holds for every radial $\Psi\in M_b(\mathbb R^3)$. Choosing $\Psi=\Phi(u)$ gives
\begin{equation}
\label{eq:rad3}
\int_{\mathbb{R}^3}\frac{|\nabla\Phi(u)|^2}{\sqrt{1-|\nabla\Phi(u)|^2/b^2}}\,{\rm d}x=-4\pi\int_{\mathbb{R}^3}(\omega+\Phi(u))\Phi(u)u^2\,{\rm d}x.
\end{equation}
Combining~\eqref{eq:rad3} with~\eqref{eq:var-ineq}, for every $\Psi\in M_b(\mathbb{R}^3)$ we obtain
\[
-\int_{\mathbb{R}^3}\frac{\nabla\Phi(u)\cdot\nabla\Psi}{\sqrt{1-|\nabla\Phi(u)|^2/b^2}}\,{\rm d}x\le4\pi\int_{\mathbb{R}^3}(\omega+\Phi(u))\Psi u^2\,{\rm d}x.
\]
Applying this to $\Psi$ and to $-\Psi$ yields equality. Finally, given $\psi\in D_\star(\mathbb{R}^3)$, the choice $\Psi=b\psi/\|\nabla\psi\|_\infty\in M_b(\mathbb{R}^3)$ shows that $\Phi(u)$ is a weak solution of~\eqref{eq:wsol-b2}.
\end{proof}

\section{The variational scheme and proof of the main result}\label{sec:5}
\subsection{The reduced functional}\label{sec:5.1}
From now on the two models are treated together. We write $G_u$ for the Born--Infeld energy of the model under consideration, that is $G_u=E_u$ as in~\eqref{eq:E} in the approximate case and $G_u=F_u$ as in~\eqref{eq:E2} in the exact case. In both cases $\Phi(u)$ denotes its unique minimizer and
\begin{equation}\label{eq:G0}
G_u(0)=\frac{\omega^2}{2}\|u\|_2^2.
\end{equation}
We define the reduced functional $J\colon H^1(\mathbb{R}^3)\to\mathbb{R}$, as
\begin{equation}
\label{eq:J}
J(u)\coloneqq\frac{1}{2}\mathcal{B}_\alpha(u,u)+\frac{m^2}{2}\|u\|_2^2-\frac{1}{p}\|u\|_p^p-G_u(\Phi(u)).
\end{equation}
We set
\[
H^1_r(\mathbb{R}^3)\coloneqq\{u\in H^1(\mathbb{R}^3):u_R=u\text{ for all $R\in\operatorname{O}(3)$}\},
\]
where $u_R$ is defined in~\eqref{eq:vR}.

\begin{lemma}\label{lem:J-C1}
Let $\beta>0$ (resp.\ $b>0$). Then $J\in C^1(H^1(\mathbb{R}^3))$ and for all $u,v\in H^1(\mathbb{R}^3)$ we have
\begin{equation}\label{eq:der-J}
J'(u)[v]=\mathcal{B}_\alpha(u,v)+m^2\int_{\mathbb{R}^3}uv\,{\rm d}x-\int_{\mathbb{R}^3}|u|^{p-2}uv\,{\rm d}x-\int_{\mathbb{R}^3}(\omega+\Phi(u))^2uv\,{\rm d}x.
\end{equation}
Moreover, the functional $J$ is even and $J(0)=0$. Finally,
\renewcommand{\labelenumi}{{\roman{enumi})}}
\begin{enumerate}
\item in the approximate case, a pair $(u,\phi)\in H^1(\mathbb{R}^3)\times D(\mathbb{R}^3)$ is a weak solution of the system if and only if $\phi=\Phi(u)$ and $u$ is a critical point of $J$;
\item in the exact case, a pair $(u,\phi)\in H^1_r(\mathbb{R}^3)\times K_b(\mathbb{R}^3)$ is a weak solution of the system if and only if $\phi=\Phi(u)$ and $u$ is a critical point of $J$.
\end{enumerate}
\end{lemma}

\begin{proof}
Write $J=J_1-G_{(\cdot)}(\Phi(\cdot))$, where $J_1(u)=\frac{1}{2}\mathcal{B}_\alpha(u,u)+\frac{m^2}{2}\|u\|_2^2-\frac{1}{p}\|u\|_p^p$ is of class $C^1$ with the expected derivative. It thus suffices to show that, for all $u,v\in H^1(\mathbb{R}^3)$,
\begin{equation}\label{eq:lim-G}
\lim_{t\to0}\frac{G_{u+t v}(\Phi(u+t v))-G_u(\Phi(u))}{t}=\int_{\mathbb{R}^3}(\omega+\Phi(u))^2uv\,{\rm d}x.
\end{equation}
On the one hand, since $\Phi(u+t v)$ minimizes $G_{u+t v}$,
\begin{align*}
&G_{u+t v}(\Phi(u+t v))-G_u(\Phi(u))-t\int_{\mathbb{R}^3}(\omega+\Phi(u))^2uv\,{\rm d}x\\
&\le G_{u+t v}(\Phi(u))-G_u(\Phi(u))-t\int_{\mathbb{R}^3}(\omega+\Phi(u))^2uv\,{\rm d}x=\frac{t^2}{2}\int_{\mathbb{R}^3}(\omega+\Phi(u))^2v^2\,{\rm d}x=o(t),
\end{align*}
as $t\to 0$,
where we used that $G_{u+t v}(\Phi(u))-G_u(\Phi(u))=\frac{1}{2}\int(\omega+\Phi(u))^2[(u+t v)^2-u^2]\,{\rm d}x$, the gradient terms cancelling. On the other hand, since $\Phi(u)$ minimizes $G_u$, we get
\begin{align*}
&G_{u+t v}(\Phi(u+t v))-G_u(\Phi(u))-t\int_{\mathbb{R}^3}(\omega+\Phi(u))^2uv\,{\rm d}x\\
&\ge G_{u+t v}(\Phi(u+t v))-G_u(\Phi(u+t v))-t\int_{\mathbb{R}^3}(\omega+\Phi(u))^2uv\,{\rm d}x\\
&=\frac{t^2}{2}\int_{\mathbb{R}^3}(\omega+\Phi(u+t v))^2v^2\,{\rm d}x+t\int_{\mathbb{R}^3}(\Phi(u+t v)-\Phi(u))(2\omega+\Phi(u+t v)+\Phi(u))uv\,{\rm d}x.
\end{align*}
We claim that both terms on the right-hand side are $o(t)$ as $t\to 0$. Indeed, by Lemma~\ref{lem:Phiu} in the approximate case, and by Lemma~\ref{lem:Phiu2} in the exact case, we have
\begin{equation}\label{eq:u+v}
\|\omega+\Phi(u+t v)\|_\infty\le \omega\quad\text{for all $t\in\mathbb{R}$}.
\end{equation}
Hence
\[
\frac{t^2}{2}\int_{\mathbb{R}^3}(\omega+\Phi(u+t v))^2v^2\,{\rm d}x\le\frac{t^2\omega^2}{2}\|v\|_2^2=o(t)\quad\text{as $t\to0$}.
\]
Next, we show that
\[
\int_{\mathbb{R}^3}(\Phi(u+t v)-\Phi(u))(2\omega+\Phi(u+t v)+\Phi(u))uv\,{\rm d}x\to0\quad\text{as $t\to0$}.
\]
In both cases, the continuity properties of $\Phi$ give $\Phi(u+tv)\to\Phi(u)$ locally uniformly as $t\to 0$. Since $-\omega\le\Phi(u+tv),\Phi(u)\le 0$, the integrand is bounded in absolute value by $2\omega^2|uv|\in L^1(\mathbb R^3)$, so dominated convergence applies. Therefore both terms are $o(t)$, which proves~\eqref{eq:lim-G} and~\eqref{eq:der-J}. The continuity of this derivative follows from the continuity properties of $\Phi$, the bound $-\omega\le \Phi(u)\le 0$, and dominated convergence. Hence $J\in C^1(H^1(\mathbb{R}^3))$.

Finally, if $(u,\phi)$ is a weak solution then $\phi$ solves the field equation, hence $\phi=\Phi(u)$ by Lemma~\ref{lem:Phiu} (approximate case) or Proposition~\ref{prop:uniq} (exact case), and~\eqref{eq:wsol-a1} together with~\eqref{eq:der-J} says exactly that $u$ is a critical point of $J$. Conversely, if $u$ is a critical point of $J$, then $(u,\Phi(u))$ solves the field equation, by Lemma~\ref{lem:Phiu} in the approximate case, and by Proposition~\ref{prop:BI-weak-sol} in the exact case provided $u$ is radial, and~\eqref{eq:der-J} gives~\eqref{eq:wsol-a1}.
\end{proof}

To find critical points of $J$ we restrict it to the radial subspace $H^1_r(\mathbb{R}^3)$ defined above.

\begin{lemma}\label{lem:sym-crit}
Let $\beta>0$ (resp.\ $b>0$). Then $u\in H^1_r(\mathbb{R}^3)$ is a critical point of $J|_{H^1_r(\mathbb{R}^3)}$ if and only if it is a critical point of $J$.
\end{lemma}

\begin{proof}
Every critical point of $J$ in $H^1_r(\mathbb{R}^3)$ is clearly a critical point of $J|_{H^1_r(\mathbb{R}^3)}$. For the converse, by the Principle of Symmetric Criticality~\cite{Palais}, it suffices to show that $J$ is invariant under $\operatorname{O}(3)$. For $R\in\operatorname{O}(3)$ we have $\Phi(u_R)=(\Phi(u))_R$, since $(\Phi(u))_R$ minimizes $G_{u_R}$. Hence $G_{u_R}(\Phi(u_R))=G_{u_R}((\Phi(u))_R)=G_u(\Phi(u))$. As $\|\nabla u_R\|_2=\|\nabla u\|_2$, $\|(-\Delta)^{\frac{s}{2}}u_R\|_2=\|(-\Delta)^{\frac{s}{2}}u\|_2$ and $\|u_R\|_q=\|u\|_q$ for every $q$, each term of~\eqref{eq:J} is $\operatorname{O}(3)$--invariant, so $J(u_R)=J(u)$. Indeed, for all $R\in \operatorname{O}(3)$ we have
\begin{align*}
J(u_R)&=\frac{1}{2}\|\nabla u_R\|_2^2+\frac{\alpha}{2}\|(-\Delta)^\frac{s}{2}u_R\|_2^2+\frac{m^2}{2}\|u_R\|_2^2-\frac{1}{p}\|u_R\|_p^p-G_{u_R}(\Phi(u_R))\\
&=\frac{1}{2}\|\nabla u\|_2^2+\frac{\alpha}{2}\|(-\Delta)^\frac{s}{2}u\|_2^2+\frac{m^2}{2}\|u\|_2^2-\frac{1}{p}\|u\|_p^p-G_{u_R}((\Phi(u))_R)\\
&=\frac{1}{2}\|\nabla u\|_2^2+\frac{\alpha}{2}\|(-\Delta)^\frac{s}{2}u\|_2^2+\frac{m^2}{2}\|u\|_2^2-\frac{1}{p}\|u\|_p^p-G_u(\Phi(u))=J(u),
\end{align*}
which allows us to apply the Principle of Symmetric Criticality and concludes the proof.
\end{proof}

On the whole space the embedding $H^1(\mathbb{R}^3)\hookrightarrow L^p(\mathbb{R}^3)$ is continuous but, for $p\in(2,6)$, not compact, the standard obstruction to the Palais--Smale condition on $\mathbb{R}^3$. Compactness is recovered on the radial subspace, by the following classical result~\cite[Theorem~A.I$'$]{BL}, which will be used to verify the Palais--Smale condition (Proposition~\ref{prop:PS}).

\begin{proposition}\label{prop:compact}
The space $H^1_r(\mathbb{R}^3)$ is compactly embedded in $L^p(\mathbb{R}^3)$ for all $p\in (2,6)$.
\end{proposition}

\subsection{Proof of the main result}\label{subsec:proof}
By Lemmas~\ref{lem:J-C1} and~\ref{lem:sym-crit}, it suffices to find infinitely many critical points of $J|_{H^1_r(\mathbb{R}^3)}$. To reach this goal, since $J\in C^1$ is even and $J(0)=0$, we shall use the well-known equivariant Mountain Pass Theorem. We now briefly recall it.

\begin{definition}\label{def:PS}
A functional $I\in C^1(X)$ on a Banach space $X$ satisfies the Palais--Smale $(PS)$ condition if every sequence $(u_n)_n\subset X$ with $(I(u_n))_n$ bounded and $I'(u_n)\to0$ in $X'$ as $n\to\infty$ admits a convergent subsequence.
\end{definition}

\begin{theorem}[{\cite[Theorem~9.12]{Rabinowitz86}}] \label{thm:Rabinowitz}
Let $I$ be an even $C^1$ functional on an infinite dimensional Banach space $X$, with $I(0)=0$. Assume $X=X_1\oplus X_2$ with $\dim X_1<\infty$, and that:
\begin{itemize}
\item[$(i)$] there exist $\delta,\varrho>0$ such that $I(u)\ge\delta$ for all $u\in X_2$ with $\|u\|_X=\varrho$;
\item[$(ii)$] for every finite dimensional subspace $Y\subset X$ there is $R=R(Y)>0$ such that $I(u)\le0$ for all $u\in Y$ with $\|u\|_X\ge R$;
\item[$(iii)$] $I$ satisfies the $(PS)$ condition.
\end{itemize}
Then $I$ has an unbounded sequence of positive critical values.
\end{theorem}

We first check the geometric assumptions, which are common to the two models.

\begin{proposition}\label{prop:geometry}
Let $\beta>0$ (resp.\ $b>0$). Assume
\[
m^2>\omega^2\quad\text{and}\quad\alpha>-\alpha_0(s,m^2-\omega^2).
\]
Then
\begin{itemize}
\item[$(i)$] there exist $\delta,\varrho>0$ such that $J(u)\ge\delta$ for all $u\in H^1_r(\mathbb{R}^3)$ with $\|u\|_{H^1}=\varrho$;
\item[$(ii)$] for every finite dimensional subspace $Y\subset H^1_r(\mathbb{R}^3)$ there is $R=R(Y)>0$ such that $J(u)\le0$ for $u\in Y$ with $\|u\|_{H^1}\ge R$.
\end{itemize}
\end{proposition}

\begin{proof}
Since $\Phi(u)$ minimizes $G_u$, by~\eqref{eq:G0} we have
\[
G_u(\Phi(u))\le G_u(0)=\frac{\omega^2}{2}\|u\|_2^2,
\]
so by~\eqref{eq:interp}, for all $\varepsilon>0$ and $u\in H^1_r(\mathbb{R}^3)$,
\begin{align*}
J(u)&\ge\frac{1}{2}\mathcal{B}_\alpha(u,u)+\frac{m^2-\omega^2}{2}\|u\|_2^2-\frac{1}{p}\|u\|_p^p\\
&\ge\frac{1}{2}(1-\alpha^-s\varepsilon)\|\nabla u\|_2^2+\frac{1}{2}\left(m^2-\omega^2-\alpha^-(1-s)\varepsilon^{-\frac{s}{1-s}}\right)\|u\|_2^2-\frac{1}{p}\|u\|_p^p.
\end{align*}
The system
\begin{equation}\label{eq:systema}
1-\alpha^-s\varepsilon>0,\qquad m^2-\omega^2-\alpha^-(1-s)\varepsilon^{-\frac{s}{1-s}}>0
\end{equation}
is trivially satisfied if $\alpha^-=0$. Otherwise it holds for some $\varepsilon>0$ as soon as
\[
\frac{(1-s)^{\frac{1-s}{s}}(\alpha^-)^{\frac{1-s}{s}}}{(m^2-\omega^2)^{\frac{1-s}{s}}}<\frac{1}{\alpha^-s},
\]
namely $\alpha^-<\alpha_0(s,m^2-\omega^2)$,
which is our hypothesis. Hence there is $\varepsilon_0>0$ with
\begin{equation}\label{eq:c1c2}
c_1\coloneqq1-\alpha^-s\varepsilon_0>0,\qquad c_2\coloneqq m^2-\omega^2-\alpha^-(1-s)\varepsilon_0^{-\frac{s}{1-s}}>0.
\end{equation}
By~\eqref{eq:emb-H}, for $\|u\|_{H^1}=\varrho$,
\[
J(u)\ge\frac{1}{2}\min\{c_1,c_2\}\varrho^2-\frac{C_p^p}{p}\varrho^p=\varrho^2\left(\frac{1}{2}\min\{c_1,c_2\}-\frac{C_p^p}{p}\varrho^{p-2}\right)>0,
\]
provided
\[
\varrho<\left(\frac{p\min \{c_1,c_2\}}{2C_p^p}\right)^\frac{1}{p-2},
\]
which proves $(i)$. For $(ii)$, since $G_u(\Phi(u))\ge0$, by~\eqref{eq:Deltas} with $\varepsilon=1$,
\[
J(u)\le\frac{1}{2}(1+|\alpha|s)\|\nabla u\|_2^2+\frac{1}{2}(m^2+|\alpha|(1-s))\|u\|_2^2-\frac{1}{p}\|u\|_p^p\to-\infty
\]
as $\|u\|_{H^1}\to\infty$ on the finite dimensional space $Y$, since $p>2$ and all norms on $Y$ are equivalent.
\end{proof}

It remains to verify the Palais--Smale condition.

\begin{proposition}
\label{prop:PS}
Let $\beta>0$ (resp. $b>0$). Recall that
\[
\Lambda=\max\left\{\frac{2}{p-2},1\right\},
\]
and assume that
\[
m^2>\Lambda\omega^2\quad\text{and}\quad\alpha>-\alpha_0(s,m^2-\Lambda\omega^2).
\]
Then $J|_{H^1_r(\mathbb{R}^3)}$ satisfies the $(PS)$ condition.
\end{proposition}

\begin{proof}
Let $(u_n)_n\subset H^1_r(\mathbb{R}^3)$ be a $(PS)$ sequence. We first prove that, in both models, the following inequality holds:
\begin{align}
&pJ(u_n)-J'(u_n)[u_n]\nonumber\\
&\ge\left(\frac{p}{2}-1\right)\left[\mathcal{B}_\alpha(u_n,u_n)+(m^2-\omega^2)\|u_n\|_2^2\right]-\left(\frac{p}{2}-2\right)\int_{\mathbb{R}^3}\omega\Phi(u_n)u_n^2\,{\rm d}x.\label{eq:PSstar}
\end{align}

\paragraph{\textbf{Approximate case}}

By~\eqref{eq:J} and~\eqref{eq:der-J} we have
\[
\begin{aligned}
pJ(u_n)-J'(u_n)[u_n]=&\left(\frac{p}{2}-1\right)\left[\mathcal{B}_\alpha(u_n,u_n)+m^2\|u_n\|_2^2\right]-\frac{p}{8\pi}\|\nabla\Phi(u_n)\|_2^2\\&-\frac{p\beta}{16\pi}\|\nabla\Phi(u_n)\|_4^4-\left(\frac{p}{2}-1\right)\int_{\mathbb{R}^3}(\omega+\Phi(u_n))^2u_n^2\,{\rm d}x.
\end{aligned}
\]
Testing~\eqref{eq:wsol-a2} with $\psi=\Phi(u_n)$ gives
\[
-\frac{p}{8\pi}\|\nabla\Phi(u_n)\|_2^2-\frac{p\beta}{8\pi}\|\nabla\Phi(u_n)\|_4^4=\frac{p}{2}\int_{\mathbb{R}^3}(\omega+\Phi(u_n))\Phi(u_n)u_n^2\,{\rm d}x,
\]
and so
\begin{align*}
pJ(u_n)-J'(u_n)[u_n]&=\left(\frac{p}{2}-1\right)\left(\mathcal{B}_\alpha(u_n,u_n)+m^2\|u_n\|_2^2\right)+\frac{p\beta}{16\pi}\|\nabla\Phi(u_n)\|_4^4\\
&\quad+\frac{p}{2}\int_{\mathbb{R}^3}(\omega+\Phi(u_n))\Phi(u_n)u_n^2\,{\rm d}x-\left(\frac{p}{2}-1\right)\int_{\mathbb{R}^3}(\omega+\Phi(u_n))^2u_n^2\,{\rm d}x.
\end{align*}
By dropping the nonnegative terms
$\frac{p\beta}{16\pi}\|\nabla\Phi(u_n)\|_4^4$ and $\int_{\mathbb{R}^3}\Phi(u_n)^2u_n^2\,{\rm d}x$, the last one
arising from the algebraic identity
\[
\frac{p}{2}(\omega+\Phi)\Phi-\left(\frac{p}{2}-1\right)(\omega+\Phi)^2=\Phi^2-\left(\frac{p}{2}-2\right)\omega\Phi-\left(\frac{p}{2}-1\right)\omega^2,
\]
we obtain~\eqref{eq:PSstar}.

\paragraph{\textbf{Exact case}}

By~\eqref{eq:J} and~\eqref{eq:der-J}, we have
\[
\begin{aligned}
pJ(u_n)-J'(u_n)[u_n]&=\left(\frac{p}{2}-1\right)\left[\mathcal{B}_\alpha(u_n,u_n)+m^2\|u_n\|_2^2\right]-\frac{pb^2}{4\pi}\int_{\mathbb{R}^3}\left(1-\sqrt{1-\frac{|\nabla\Phi(u_n)|^2}{b^2}}\right){\rm d}x\\
&\quad-\left(\frac{p}{2}-1\right)\int_{\mathbb{R}^3}(\omega+\Phi(u_n))^2u_n^2\,{\rm d}x.
\end{aligned}
\]
Using 
\[
b^2\left(1-\sqrt{1-t}\right)=\frac{b^2 t}{1+\sqrt{1-t}},\quad\text{with }t=\frac{|\nabla\Phi(u_n)|^2}{b^2},
\]
we get
\[
-\frac{pb^2}{4\pi}\int_{\mathbb{R}^3}\left(1-\sqrt{1-\frac{|\nabla\Phi(u_n)|^2}{b^2}}\right){\rm d}x=-\frac{p}{4\pi}\displaystyle\int_{\mathbb{R}^3}\frac{|\nabla\Phi(u_n)|^2}{1+\sqrt{1-|\nabla\Phi(u_n)|^2/b^2}}\,{\rm d}x.
\]
Now, since
\[
-\frac{|\xi|^2}{1+\sqrt{1-|\xi|^2}}\ge-\frac{|\xi|^2}{2\sqrt{1-|\xi|^2}}\quad\text{for every $\xi$ with $|\xi|<1$},
\]
we obtain
\[
\begin{aligned}
pJ(u_n)-J'(u_n)[u_n]&\ge\left(\frac{p}{2}-1\right)\left[\mathcal{B}_\alpha(u_n,u_n)+m^2\|u_n\|_2^2\right]-\frac{p}{8\pi}\int_{\mathbb{R}^3}\frac{|\nabla\Phi(u_n)|^2}{\sqrt{1-|\nabla\Phi(u_n)|^2/b^2}}\,{\rm d}x\\
&\quad-\left(\frac{p}{2}-1\right)\int_{\mathbb{R}^3}(\omega+\Phi(u_n))^2u_n^2\,{\rm d}x.
\end{aligned}
\]
Testing~\eqref{eq:wsol-b2} with $\psi=\Phi(u_n)$ gives
\[
-\frac{p}{8\pi}\int_{\mathbb{R}^3}\frac{|\nabla\Phi(u_n)|^2}{
\sqrt{1-|\nabla\Phi(u_n)|^2/b^2}}\,{\rm d}x=\frac{p}{2}\int_{\mathbb{R}^3}(\omega+\Phi(u_n))\Phi(u_n)u_n^2\,{\rm d}x,
\]
whence
\begin{align*}
pJ(u_n)-J'(u_n)[u_n]&\ge\left(\frac{p}{2}-1\right)\left(\mathcal{B}_\alpha(u_n,u_n)+m^2\|u_n\|_2^2\right)\\ &\quad-\frac{p\omega}{2}\int_{\mathbb{R}^3}(\omega+\Phi(u_n))u_n^2\,{\rm d}x+\int_{\mathbb{R}^3}(\omega+\Phi(u_n))^2u_n^2\,{\rm d}x\\
&\ge\left(\frac{p}{2}-1\right)\left(\mathcal{B}_\alpha(u_n,u_n)+(m^2-\omega^2)\|u_n\|_2^2\right)-\left(\frac{p}{2}-2\right)\int_{\mathbb{R}^3}\omega\Phi(u_n)u_n^2\,{\rm d}x,
\end{align*}
which is nothing but~\eqref{eq:PSstar}.

We now exploit~\eqref{eq:PSstar}, distinguishing two cases.

\paragraph{\textbf{Case (a): $p\in[4,6)$}}
Here $\Lambda=1$, so $m^2>\omega^2$. Since $\frac{p}{2}-2\ge0$ and $\Phi(u_n)\le0$, the last term in~\eqref{eq:PSstar} is nonnegative, hence
\[
pJ(u_n)-J'(u_n)[u_n]\ge\left(\frac{p}{2}-1\right)\left[\mathcal{B}_\alpha(u_n,u_n)+(m^2-\omega^2)\|u_n\|_2^2\right].
\]
By~\eqref{eq:interp}, choosing $\varepsilon_0$ as in~\eqref{eq:c1c2} (admissible since $\alpha>-\alpha_0(s,m^2-\omega^2)$), we get
\begin{equation}\label{eq:casea}
pJ(u_n)-J'(u_n)[u_n]\ge\left(\frac{p}{2}-1\right)\min\{c_1,c_2\}\|u_n\|_{H^1}^2.
\end{equation}

\paragraph{\textbf{Case (b): $p\in(2,4)$}} Here $\Lambda=\frac{2}{p-2}$, so
\[
m^2>\frac{2\omega^2}{p-2}.
\]
Now $\frac{p}{2}-2<0$ and $\Phi(u_n)\ge-\omega$, so
\[
-\left(\frac{p}{2}-2\right)\int_{\mathbb{R}^3}\omega\Phi(u_n)u_n^2\,{\rm d}x\ge\omega^2\left(\frac{p}{2}-2\right)\|u_n\|_2^2,
\]
and by~\eqref{eq:interp},
\begin{align}
pJ(u_n)-J'(u_n)[u_n]&\ge\left(\frac{p}{2}-1\right)(1-\alpha^-s\varepsilon)\|\nabla u_n\|_2^2\nonumber\\
&\quad+\left(\frac{p}{2}-1\right)\left(m^2-\alpha^-(1-s)\varepsilon^{-\frac{s}{1-s}}\right)\|u_n\|_2^2-\omega^2\|u_n\|_2^2.
\label{eq:PSlb2}
\end{align}
The system
\begin{equation}\label{eq:systemb}
1-\alpha^-s\varepsilon>0,\qquad\left(\frac{p}{2}-1\right)\left(m^2-\alpha^-(1-s)\varepsilon^{-\frac{s}{1-s}}\right)>\omega^2
\end{equation}
has a solution for some $\varepsilon=\varepsilon_1>0$ exactly when
\[
\alpha^-<\alpha_0\left(s,m^2-\frac{2\omega^2}{p-2}\right),
\]
which holds by hypothesis. By setting
\begin{equation}\label{eq:d1d2}
d_1\coloneqq\left(\frac{p}{2}-1\right)(1-\alpha^-s\varepsilon_1)>0,\qquad d_2\coloneqq\left(\frac{p}{2}-1\right)\left(m^2-\alpha^-(1-s)\varepsilon_1^{-\frac{s}{1-s}}\right)-\omega^2>0,
\end{equation}
from~\eqref{eq:PSlb2} we obtain
\begin{equation}\label{eq:caseb}
pJ(u_n)-J'(u_n)[u_n]\ge\min\{d_1,d_2\}\|u_n\|_{H^1}^2.
\end{equation}

Since $(J(u_n))_n$ is bounded and $J'|_{H^1_r}(u_n)\to0$, there are $K_1,K_2>0$ with $J(u_n)\le K_1$ and $|J'(u_n)[u_n]|\le K_2\|u_n\|_{H^1}$. By~\eqref{eq:casea} and~\eqref{eq:caseb}, $pK_1+K_2\|u_n\|_{H^1}\ge c_3\|u_n\|_{H^1}^2$ for a suitable $c_3>0$, so $(u_n)_n$ is bounded in $H^1_r(\mathbb{R}^3)$. Up to a subsequence, not relabeled, there is a function $u \in H^1_r(\mathbb{R}^3)$ such that, as $n\to\infty$
\begin{align*}
&u_n\rightharpoonup u\quad\text{in $H^1(\mathbb{R}^3)$},\\
&u_n\to u\quad\text{in $L^p(\mathbb{R}^3)$ for all $p\in (2,6)$},\\
&u_n\to u\quad\text{a.e.\ in $\mathbb{R}^3$}.
\end{align*}
We are going to show that $u_n\to u$ in $H^1(\mathbb{R}^3)$ as $n\to\infty$. 
We first show that 
\begin{equation}\label{eq:st-Phi}
\Phi(u_n)\to \Phi(u)\quad\text{in $L^\infty_{\mathrm{loc}}(\mathbb{R}^3)$ as $n\to\infty$},
\end{equation}
by distinguishing between the approximate case and the exact case.

\paragraph{\textbf{Approximate case}} Testing~\eqref{eq:wsol-a2} with $\psi=\Phi(u_n)$ and using $-\omega\le\Phi(u_n)\le0$, we get
\[
\begin{aligned}
\|\nabla\Phi(u_n)\|_2^2+\beta\|\nabla\Phi(u_n)\|_4^4 &=-4\pi\omega\int_{\mathbb{R}^3}\Phi(u_n)u_n^2\,{\rm d}x-4\pi\int_{\mathbb{R}^3}\Phi(u_n)^2u_n^2\,{\rm d}x\\ &\le-4\pi\omega\int_{\mathbb{R}^3}\Phi(u_n)u_n^2\,{\rm d}x\le4\pi\omega^2\|u_n\|_2^2,
\end{aligned}
\]
so $(\Phi(u_n))_n$ is bounded in $D(\mathbb{R}^3)$. Arguing as in the continuity part of Lemma~\ref{lem:Phiu}, with $u_n,u$ in place of the approximating sequence, one finds
\begin{align*}
&\|\nabla \Phi(u_n)-\nabla\Phi(u)\|_2^2+\frac{\beta}{4}\|\nabla \Phi(u_n)-\nabla\Phi(u)\|_4^4\\
&\le 4\pi\int_{\mathbb{R}^3}(\omega+\Phi(u))(\Phi(u_n)- \Phi(u))(u^2-u_n^2)\,{\rm d}x\\
&\le 4\pi\omega\|\Phi(u_n)- \Phi(u)\|_6\|u+u_n\|_2\|u-u_n\|_3\le C\|u-u_n\|_3,
\end{align*}
for a constant $C>0$ independent of $n$. Since $u_n\to u$ in $L^3(\mathbb{R}^3)$ as $n\to\infty$, we deduce
\[
\Phi(u_n)\to \Phi(u)\quad\text{in $D(\mathbb{R}^3)$ as $n\to\infty$},
\]
which yields~\eqref{eq:st-Phi}, in view of Proposition~\ref{prop:embedding}.

\paragraph{\textbf{Exact case}} From~\eqref{eq:L2-bound} there exists a constant $C>0$ such that
\begin{align*}
\|\nabla\Phi(u_n)\|_2^2\le 4\pi\omega^2\|u_n\|_2^2,
\end{align*}
which implies that $(\Phi(u_n))_n$ is bounded in $D^{1,2}(\mathbb{R}^3)$. Hence, there exists $\phi\in D^{1,2}(\mathbb{R}^3)$ and a not relabeled subsequence such that
\[
\Phi(u_n)\rightharpoonup \phi \quad\text{in $D^{1,2}(\mathbb{R}^3)$ as $n\to\infty$}.
\]
By Lemma~\ref{lem:weak-closed} we get $\phi\in M_b(\mathbb{R}^3)$. By Lemma~\ref{lem:conv-unif} we derive
\begin{equation}\label{eq:Linfty-aa}
\Phi(u_n)\to \phi \quad\text{in $L^\infty_{\mathrm{loc}}(\mathbb{R}^3)$ as $n\to\infty$}.
\end{equation}
Let us show that $\phi=\Phi(u)$. For all $n\in \mathbb{N}$ and $\psi\in M_b(\mathbb{R}^3)$ we have that
\[
F_{u_n}(\Phi(u_n))\le F_{u_n}(\psi),
\]
which yields, simplifying the term $\frac{\omega^2}{2}\int_{\mathbb{R}^3}u_n^2\,{\rm d}x$,
\begin{align*}
&\frac{b^2}{4\pi}\int_{\mathbb{R}^3}\left(1-\sqrt{1-\frac{|\nabla\Phi(u_n)|^2}{b^2}}\right){\rm d}x+\int_{\mathbb{R}^3}\omega\Phi(u_n)u_n^2\,{\rm d}x+\frac{1}{2}\int_{\mathbb{R}^3}\Phi(u_n)^2u_n^2\,{\rm d}x\\
&\le \frac{b^2}{4\pi}\int_{\mathbb{R}^3}\left(1-\sqrt{1-\frac{|\nabla\psi|^2}{b^2}}\right){\rm d}x+\int_{\mathbb{R}^3}\omega\psi u_n^2\,{\rm d}x+\frac{1}{2}\int_{\mathbb{R}^3}\psi^2u_n^2\,{\rm d}x.
\end{align*}
Since $u_n^2\to u^2$ in $L^{6/5}(\mathbb{R}^3)$ as $n\to\infty$ and $\psi,\psi^2\in L^6(\mathbb{R}^3)$, H\"older's inequality gives
\begin{align*}
&\lim_{n\to\infty}\left[\frac{b^2}{4\pi}\int_{\mathbb{R}^3}\left(1-\sqrt{1-\frac{|\nabla\psi|^2}{b^2}}\right){\rm d}x+\int_{\mathbb{R}^3}\omega\psi u_n^2\,{\rm d}x+\frac{1}{2}\int_{\mathbb{R}^3}\psi^2u_n^2\,{\rm d}x\right]\\
&=\frac{b^2}{4\pi}\int_{\mathbb{R}^3}\left(1-\sqrt{1-\frac{|\nabla\psi|^2}{b^2}}\right){\rm d}x+\int_{\mathbb{R}^3}\omega\psi u^2\,{\rm d}x+\frac{1}{2}\int_{\mathbb{R}^3}\psi^2u^2\,{\rm d}x,
\end{align*}
while weak lower semicontinuity of the gradient term, as in Lemma~\ref{lem:weak-semi}, yields the following inequality
\begin{align*}
&\frac{b^2}{4\pi}\int_{\mathbb{R}^3}\left(1-\sqrt{1-\frac{|\nabla\phi|^2}{b^2}}\right){\rm d}x+\int_{\mathbb{R}^3}\omega\phi u^2\,{\rm d}x+\frac{1}{2}\int_{\mathbb{R}^3}\phi^2u^2\,{\rm d}x\\
&\le \liminf_{n\to\infty}\left[\frac{b^2}{4\pi}\int_{\mathbb{R}^3}\left(1-\sqrt{1-\frac{|\nabla\Phi(u_n)|^2}{b^2}}\right){\rm d}x+\int_{\mathbb{R}^3}\omega\Phi(u_n)u_n^2\,{\rm d}x+\frac{1}{2}\int_{\mathbb{R}^3}\Phi(u_n)^2u_n^2\,{\rm d}x\right].
\end{align*}
Here the potential terms converge by $u_n^2 \to u^2$ in $L^{6/5}(\mathbb{R}^3)$ as $n\to\infty$, the uniform $L^6$ bound on $\Phi(u_n)$, local uniform convergence, and $-\omega\le \Phi(u_n)\le 0$. Thus, by adding the term $\frac{\omega^2}{2}\int_{\mathbb{R}^3}u^2\,{\rm d}x$, we conclude
\[
F_u(\phi)\le F_u(\psi)\quad\text{for all $\psi\in M_b(\mathbb{R}^3)$}.
\]
By the uniqueness of the minimizer, we conclude that $\phi=\Phi(u)$, which gives~\eqref{eq:st-Phi}.

As a consequence of~\eqref{eq:st-Phi}, we derive that for all $v\in H^1(\mathbb{R}^3)$
\[
J'(u_n)[v]\to J'(u)[v]\quad\text{as $n\to\infty$}.
\]
Since $(u_n)_n$ is a Palais--Smale sequence for the radial restriction, $u$ is a critical point of $J|_{H^1_r(\mathbb{R}^3)}$. By Lemma~\ref{lem:sym-crit}, $u$ is a critical point of $J$. Thus
\begin{equation}\label{eq:der-J2-0}
J'(u)=\mathcal L_\alpha u+(m^2-\omega^2) u-2\omega\Phi(u)u-\Phi(u)^2u-|u|^{p-2}u=0\quad\text{in $H^{-1}(\mathbb{R}^3)$}.
\end{equation}
We now show that $u_n\to u$ in $H^1(\mathbb{R}^3)$ as $n\to\infty$. By~\eqref{eq:der-J}, for all $n\in\mathbb{N}$ we have
\begin{equation}\label{eq:der-J2un}
J'(u_n)=\mathcal{L}_\alpha u_n + (m^2-\omega^2)u_n-2\omega\Phi(u_n)u_n-\Phi(u_n)^2u_n-|u_n|^{p-2}u_n\quad\text{in $H^{-1}(\mathbb{R}^3)$}.
\end{equation}
Since
\[
\alpha> -\alpha_0\left(s,m^2-\max\left\{\frac{2}{p-2},1\right\}\omega^2\right)\ge -\alpha_0(s,m^2-\omega^2),
\]
the operator $\mathcal{M}_\alpha \colon H^1(\mathbb{R}^3)\to H^{-1}(\mathbb{R}^3)$, defined for $u\in H^1(\mathbb{R}^3)$ as
\[
\mathcal{M}_\alpha u\coloneqq \mathcal{L}_\alpha u + (m^2-\omega^2)u\quad\text{in $H^{-1}(\mathbb{R}^3)$},
\]
is coercive, since
\[
\langle\mathcal{M}_\alpha u,u\rangle_{H^{-1}\times H^1}\ge \min\{c_1,c_2\}\|u\|_{H^1}^2\quad\text{for all $u\in H^1(\mathbb{R}^3)$},
\]
where $c_1,c_2>0$ are the constants defined in~\eqref{eq:c1c2}. Moreover, as $n\to\infty$,
\[
|u_n|^{p-2}u_n\to |u|^{p-2}u\quad\text{in $L^{p'}(\mathbb{R}^3)$}.
\]

By Sobolev's inequality~\eqref{eq:Sob}, H\"{o}lder's inequality $\|uv\|_2\le \|u\|_6\|v\|_3,$ and the Dominated Convergence Theorem, using~\eqref{eq:st-Phi}, we obtain as $n\to\infty$
\begin{align*}
\|\Phi(u_n)u_n-\Phi(u)u\|_2&\le S\|\nabla\Phi(u_n)\|_2\|u_n-u\|_3+\|\Phi(u_n)u-\Phi(u)u\|_2\to 0,\\
\|\Phi(u_n)^2u_n-\Phi(u)^2u\|_2&\le S\omega\|\nabla \Phi(u_n)\|_2\|u_n-u\|_3+2\omega\|(\Phi(u_n)-\Phi(u))u\|_2\to 0.
\end{align*}
Subtracting~\eqref{eq:der-J2-0} from~\eqref{eq:der-J2un} and testing with $u_n-u$, we derive as $n\to\infty$
\begin{align*}
\min\{c_1,c_2\}\|u_n-u\|_{H^1}^2&\le |J'(u_n)[u_n-u]|+2\omega\|\Phi(u_n)u_n-\Phi(u)u\|_2\|u_n-u\|_2\\
&\quad+\|\Phi(u_n)^2u_n-\Phi(u)^2u\|_2\|u_n-u\|_2\\
&\quad+\||u_n|^{p-2}u_n-|u|^{p-2}u\|_{p'}\|u_n-u\|_p\to 0.
\end{align*}
Indeed, $(u_n-u)_n$ is bounded and radial, so the first term tends to zero by the Palais--Smale assumption on $H^1_r(\mathbb{R}^3)$. The remaining terms tend to zero by the convergences just proved. Hence $u_n\to u$ in $H^1(\mathbb{R}^3)$ as $n\to\infty$, which proves the (PS) condition.
\end{proof}

Having verified the geometric conditions (Proposition~\ref{prop:geometry}) and the Palais--Smale condition (Proposition~\ref{prop:PS}), we are now ready to apply Theorem~\ref{thm:Rabinowitz} and prove our main results.

\begin{proof}[Proof of Theorem~\ref{thm:main}]
Under the stated hypotheses, $J|_{H^1_r(\mathbb{R}^3)}$ is an even $C^1$ functional with $J(0)=0$, and by Propositions~\ref{prop:geometry} and~\ref{prop:PS} it satisfies conditions $(i)$--$(iii)$ of Theorem~\ref{thm:Rabinowitz} (with $X=H^1_r(\mathbb{R}^3)$, $X_1=\{0\}$ and $X_2=X$). Hence $J|_{H^1_r(\mathbb{R}^3)}$ possesses an unbounded sequence of critical values, and therefore infinitely many critical points $(u_n)_n\subset H^1_r(\mathbb{R}^3)$. By Lemma~\ref{lem:sym-crit} each $u_n$ is a critical point of $J$, and by Lemma~\ref{lem:J-C1} the pairs $(u_n,\Phi(u_n))$ are radially symmetric weak solutions of~\eqref{eq:KGBI} in $H^1(\mathbb{R}^3)\times D(\mathbb{R}^3)$ (approximate case, part $(i)$) and of~\eqref{eq:NKGBI} in $H^1(\mathbb{R}^3)\times K_b(\mathbb{R}^3)$ (exact case, part $(ii)$), respectively.
\end{proof}

\section*{Statements and Declarations}

\noindent {\bf Conflict of interest} The authors declare that they have no conflict of interest.

\bigskip

\noindent {\bf Acknowledgements.} The authors M.\ Caponi and E.\ Vitillaro are members of GNAMPA - Istituto Nazionale di Alta Matematica (INdAM). M.\ Caponi acknowledges the support of the INdAM - GNAMPA 2026 Project ``Analisi variazionale per operatori locali e nonlocali possibilmente singolari o degeneri'' (CUP: E53C25002010001). M.\ Caponi has been funded by the European Union - NextGenerationEU under the Italian Ministry of University and Research (MUR) National Centre for HPC, Big Data and Quantum Computing (CN\_00000013 - CUP: E13C22001000006).

\end{document}